\documentclass{svjour3_my}                     
\usepackage{graphicx}
\usepackage{mathptmx}   
\usepackage{latexsym} 
\usepackage[cp1251]{inputenc}
\usepackage{amsmath,amsbsy}
\usepackage{amsfonts,amssymb,mathrsfs,amscd}
\usepackage{verbatim}
\usepackage{eucal}
\usepackage{graphicx}
\usepackage{enumerate}
\usepackage[pdftex,unicode,colorlinks,linkcolor=blue,
citecolor=red,bookmarksopen,pdfhighlight=/N]{hyperref}

\numberwithin{equation}{section}

\newtheorem{dualtheo}{Duality Theorem}
\newtheorem{gluingtheorem}{Gluing Theorem} 
\newtheorem{PLf}{Poincar\'e\,--\,Lelong formula}
\newtheorem{PJf}{Poisson\,--\,Jensen formula}

\newtheorem{comments}{Comments}

\newcommand{\RR}{\mathbb{R}} 
\newcommand{\CC}{\mathbb{C}} 
\newcommand{\NN}{\mathbb{N}} 
\newcommand{\BB}{\mathbb{B}}  

\newcommand{\const}{{\rm const}} 
\newcommand{\pt}{{\rm pt}} 
\newcommand{\dd}{\,{\rm d}}
\DeclareMathOperator{\dist}{dist} 
\DeclareMathOperator{\clos}{clos} 
\DeclareMathOperator{\Int}{int}
\DeclareMathOperator{\Meas}{Meas}
\DeclareMathOperator{\har}{har} 
\DeclareMathOperator{\Hol}{Hol}
\DeclareMathOperator{\comp}{cmp}
\DeclareMathOperator{\Zero}{Zero} 
\DeclareMathOperator{\sbh}{sbh}
\DeclareMathOperator{\dsbh}{dsbh} 
\DeclareMathOperator{\supp}{supp} 
\DeclareMathOperator{\loc}{loc}
\DeclareMathOperator{\sgn}{sgn} 
\DeclareMathOperator{\dom}{dom}
\DeclareMathOperator{\conn}{conn}
\DeclareMathOperator{\Conn}{Conn}

\begin{document}

\title{Preorders on Subharmonic Functions and Measures\thanks{The work was supported by a Grant of the Russian Science Foundation (Project No. 18-11-00002, Bulat N. Khabibullin, sections 1--3, 7--12), and by a Grant of the Russian Foundation of Basic Research (Projects No.~19-31-90007, Enzhe B. Menshikova, sections 4--6).}
}
\subtitle{with Applications to the Distribution of Zeros of Holomorphic Functions}
%%Do you have a subtitle?\\ If so, write it here}
%%%\titlerunning{Short form of title}        % if too long for running head
\author{Bulat~N.~Khabibullin%%First Author         
\and Enzhe~B.~Menshikova%%Second Author %etc.
}


\institute{Bulat N. Khabibullin%%F. Author 
\at
Department of Mathematics and IT, Bashkir State University, Ufa, Bashkortostan, Russian Federation
 %%first address 
\\
%%Tel.: +7-917-407-86-54\\
%%Fax: +123-45-678910\\
\email{khabib-bulat@mail.ru}           %  \\
%%\and 
\\
Enzhe B. Menshikova%% S. Author 
\at
%%second address
Department of Mathematics and IT, Bashkir State University, Ufa, Bashkortostan, Russian Federation 
%%first address 
\\
%%Tel.: +7-347-2736718\\
\email{algeom@bsu.bashedu.ru}
}

\date{}

\maketitle

\begin{abstract}
Let $X$ be a class of  extended numerical  functions on a domain $D$ of $d$-dimensional Euclidean space $\mathbb R^d$, $H\subset X$. Given $u,M\in X$, we  write $u\prec_H M$ if there is a function $h\in H$ such that $u+h\leq M$ on $D$. We consider this special preorder $\prec_H$ for a pair of subharmonic functions $u, M$ on $D$ in cases where $H$ is the space of all harmonic functions on $D$ or $H$ is the convex cone of all subharmonic functions $h \not\equiv -\infty$ on $D$. Main results are  dual equivalent forms for this preorder $\prec_H$ in terms of balayage processes for Riesz measures of subharmonic functions $u$ and $M$, for Jensen and Arens\,--\,Singer (representing) measures, for potentials of these measures, and for special test functions generated by subharmonic functions on complements $D\!\setminus\!S$ of non-empty precompact subsets $S\Subset D$. Applications to holomorphic functions $f$ on a domain $D\subset \mathbb C^n$ relate to the distribution of zero sets of functions $f$ under upper restrictions $|f|\leq \exp M$ on $D$. If a domain $D\subset \mathbb C$ is a finitely connected domain with non-empty exterior or a simply connected domain with two different points
 on the boundary of $D$,  then our conditions for the distribution of zeros of $f\neq 0$ with $|f|\leq \exp M$ on $D$  are both necessary and sufficient.
\keywords{subharmonic function\and balayage\and Jensen measure\and
representing measure\and holomorphic function\and zero set\and Hausdorff measure}
\subclass{Primary: 31B05; Secondary: 32A60\and31C05\and31B15\and31B25\and31A05\and30C85\and31A15\and32E35}
\end{abstract}

\tableofcontents

\textbf{\section{Introduction}\label{int}}

Let $(X,+,0, \leq)$ be a preordered monoid equipped with a binary operation-addi\-t\-i\-on $+$ with a neutral element $0$ and a reflexive and transitive preorder $\leq$, and let $H\subset X$ be a preordered submonoid in $X$, $0\in H$. We can define a special  preorder $\prec_H$ on $X$:  
\begin{equation}\label{prec}
x'\prec_H x\quad\text{ \it if there exists $h\in H$ such that $x'+h\leq x$}. 
\end{equation} 
This preorder $\prec_H$ is weaker than the original preorder $\leq$. 
In \cite[Ch.~VI, 8]{Brelot}, \cite[\S~4.1]{ConCor}, another specific  order $\preccurlyeq_H$ of M.~Brelot for some ordered groups $(X,+,0,\leq)$ with $H\subset X^+:=\{x\in X\colon 0\leq x\}$ was considered: {\it $x'\preccurlyeq x$ if there exists $h\in H$ such that $x'+h= x$}. This specific  order $\preccurlyeq_H$ 
 is stronger than the original order $\leq$.  
Generally speaking, our special preorder $\prec_H$ is even strictly weaker than the original order $\leq$ and the specific Brelot order $\preccurlyeq_H$.
In this article,  we study only ``subharmonic'' and ``harmonic''
versions of preorder $\prec_H$: 
\begin{enumerate}
\item[{[O]}] \underline{Throughout this article}, $d\in \NN:=\{1,2, \dots \}$,
 $\NN_0:=\{0\}\cup \NN$, $\RR$ is the {\it real line\/},  $\RR^d$ is the $d$-dimensional Euclidean space $\RR^d$ with the  Euclidean norm $|x|:=\sqrt{x_1^2+\dots+x_d^2}$ of $x=(x_1,\dots ,x_d)\in \RR^d$ and the {\it distance function\/} $\dist (\cdot, \cdot)$; $O\subset \RR^d$ is a {\it non-empty open set}, and a fixed point $o\in O$ is often considered as an {\it origin point of} $O$, 
unless they are a notation for  big-$O(\cdot)$ or little-$o(\cdot)$.

\item[{[L]}] $X:=L_{\loc}^1(O)$ is the set of all extended numerical  
 functions $f\colon O\to \overline \RR$ that locally $\lambda_d$-integrable on $O$, where   $\lambda_d$ is  the {\it Lebesgue measure on\/} $\RR^d$, and $ \overline \RR:=\RR\cup \{\pm\infty\}$ is the {\it extended real line\/} in the end topology with two ends $\pm \infty$, with   the usual linear order $\leqslant$ on $\RR$ complemented by the inequalities $-\infty \leqslant r\leqslant +\infty$ for $r\in \overline{\RR}$, with the {\it positive real axis} $\RR^+:= \{r\in \RR\colon r\geqslant 0\}$.
 
\item[{[$\leq$]}] $L_{\loc}^1(O)$ is equipped with the  \textit{pointwise 
preorder:\/} $f\leq g$ {\it on $O$} for $f,g\in L_{\loc}^1(O)$ if {\it $f(x)\leqslant g(x)$  for\/ $\lambda_d$-almost every\/ $x\in O$.}

\item[{[H]}] $H:=\sbh_*(O):=\bigl\{u\in \sbh(O)\colon u\not\equiv -\infty \bigr\}\subset L^1_{\loc}(O)$, where $\sbh(O)$ is the convex cone over $\RR$ of all subharmonic (convex for $d=1$) functions on $O$  ({\it the subharmonic version}\/),  
\underline{or} $H:=\har(O):=\sbh(O)\cap \bigl(-\sbh(O)\bigr)$ is the vector space over $\RR$ of all harmonic (affine for $d=1$) functions on $O$ ({\it the harmonic version\/}). If $u\leq M$ on $O$ for $u,M\in \sbh(O)$, then $u(x)\leqslant M(x)$ for each $x\in O$ \cite{HK}, \cite{Helms}, \cite{Landkoff},  \cite{R}, \cite{Doob}. 
\end{enumerate}

Our main goal is to obtain dual equivalent conditions for the preorder \eqref{prec} in the case of an arbitrary pair of subharmonic functions $u,M\in \sbh_*(O)$  in the role of $x',x\in X$ in terms of their Riesz measures. Our dual descriptions use 
both a linear balayage and an affine balayage for Riesz measures, for Jensen measures and for their potentials in the subharmonic version $H:=\sbh_*(O)$ (see Theorem \ref{crit1} in Sec.~\ref{Sssh}), and for Riesz measures, for Arens\,--\,Singer measures,  and for their potentials in the harmonic version $H:=\har(O)$
(see Theorem \ref{crit2} in Sec.~\ref{Sshv}). 
Our interest in the preorder $\prec_H$ on subharmonic functions is largely motivated by its applications to the non-triviality of weighted spaces of holomorphic functions
\cite{Kh010}, \cite{Kha01}, \cite[\S~10]{Kha01II}, \cite{KhaRozKha19}, \cite[9.1]{KhaRozKha19}, to the description of zero sets for functions of this spaces \cite{Kha91}, \cite{Kha93}, \cite{Kha94}, 
\cite[III]{Koosis96}, \cite{Kha96}, \cite[\S\S~8, 11]{Kha01II}, \cite{Kha03MN}, \cite{Kha06}, \cite{Kha07}, \cite{Kha09}, 
 \cite{Kha12}, \cite{BaiTalKha16},  \cite{KhaKha19}, \cite{KhaKhaChe09}, \cite{KhaRoz18}, \cite[9.2, 9.3]{KhaRozKha19}, \cite{KhaTam17A}, \cite{KudKha09},  \cite{MenKha19},
to the approximation by exponential and similar systems in function spaces \cite{Kha94BM}, \cite{Kha99}, \cite{Kha12}, \cite{KhaTalKha15},
to the representation of meromorphic functions \cite{Kha96}, \cite{KhCV98}, \cite{Kha02}, \cite{Kha03MN}, \cite{Kha06}, \cite{Kha07}, \cite{Kha09}, \cite[9.4]{KhaRozKha19}, \cite{KudKha09},  etc. 

We denote by $Y^X$ the set of all functions  $f\colon X\to Y$, where values $f(x)$  is not necessarily defined for all $x \in X$, and 
the \textit{ domain of definition\/} of function $f$ is a set
\begin{equation*}
\dom f:=\bigl\{x\in X\colon \text{\it  there is $y\in Y$ such that $f(x)=y$}\bigr\}=f^{-1}(Y)\subset X.  
\end{equation*}
\begin{definition}[{\rm see \cite{KhaRozKha19}, \cite{MenKha19}, \cite{KhaKha19}; cf. \cite{BH}, \cite{Landkoff}, \cite[Ch.~9]{Meyer}, \cite[Ch.~7]{ConCor}, \cite{Brelot}, \cite{Doob}, \cite{Helms}, \cite[\S~4]{Kha01I}, \cite[\S~7]{Kha01II}}]\label{bal}
Let $\langle\cdot, \cdot \rangle\in \overline{\RR}^{X\times Y}$, i.e.,
$\dom \langle\cdot, \cdot \rangle \ni (x,y)\overset{\langle\cdot, \cdot \rangle}{\longmapsto} \langle x, y \rangle\in \overline{\RR}$. 

Let $x'\in X$, $x\in  X$, $Y'\subset Y$, and $\{x',x\}\times Y'\subset 
\dom \langle\cdot, \cdot \rangle$. We write $x'\curlyeqprec_{Y'} x$ and say that  $x$ is an {\it affine $Y'$-balayage of\/} $x'$,   if there is a constant $c\in \RR$ such that    
$\langle x', y \rangle\leqslant \langle x, y \rangle+c$
{\it for each} $y\in Y'$,  but if we can choose this constant $c:=0$, then  
 $x$ is a {\it linear $Y'$-balayage of\/} $x'$ and we write  $x'\preceq_{Y'} x$.

Let $y'\in Y$, $y\in  Y$, $X'\subset X$, and $X'\times \{y',y\}\subset 
\dom \langle\cdot, \cdot \rangle$. We write $y'\curlyeqprec_{X'} y$ and say that  $y$ is a {\it affine $X'$-balayage of $y'$,}   if  there is a constant $c\in \RR$ such that      $\langle x, y' \rangle\leqslant \langle x, y \rangle+c$
{\it for each} $x\in X'$,  but if we can choose this constant $c:=0$, then  
 $y$ is a {\it linear $X'$-balayage of\/} $y'$ and we write  $y'\preceq_{X'} y$.
\end{definition}

Let $\mathbb R^d_{\infty}:=\mathbb R^d \cup \{\infty\}$
is the {\it  Alexandroff\/} one-point {\it compactification\/} of $\mathbb R^d$, and  let $Q$ be a subset in $\RR_{\infty}^d$.
We denote by $\mathcal B (Q)$ the class of all Borel subsets in $Q$. 
The \textit{closure\/} $\clos Q$, the\textit{ interior\/} $\Int Q$  and the \textit{boundary\/} $\partial Q$ of $Q$ will always be taken relative to $\RR^d_{\infty}$. For $Q'\subset Q\subset \RR^d_{\infty}$ we write  $Q'\Subset Q$ if $\clos Q'\subset \Int Q$. 

In our article, we use the linear and affine balayage only in the following variants:
\begin{enumerate}
\item[{[X]}] $X$ is a subset of one of the classes $\Meas^+(Q)$ of all (Radon positive) {\it measures\/} $\mu$ on   $Q\in \mathcal B(\mathbb R^d_{\infty})$ \cite{Bourbaki}, 
 \cite[Appendix A]{R} \underline{or} $\Meas(Q):=\Meas^+(Q)-\Meas^+(Q)$  of {\it charges} \cite{Landkoff} \underline{or} $\Meas_{\comp}(Q)$ of charges $\mu \in \Meas(Q)$ with a compact {\it support\/} $\supp \mu \Subset Q$ \underline{or} $\Meas_{\comp}^+(Q):=\Meas_{\comp}(Q)\cap \Meas^+(Q)$ \underline{or} $\Meas^{1+}(Q)$ of {\it probability measures\/} \underline{or} $\Meas_{\comp}^{1+}(Q):=\Meas_{\comp}(Q)\cap \Meas^{1+}(Q)$ \underline{or}  Jensen or Arens\,--\,Singer (representing) measures. 
\item[{[Y]}] $Y\subset \sbh(Q)$, where $\sbh(Q)$ 
 consists of the restrictions to $Q$ of subharmonic functions on open sets $O \supset Q$; $\har(Q):=\sbh(Q)\cap \bigl(-\sbh(Q)\bigr)$. 
\item[{[$\langle\,, \, \rangle$]}] the function $\langle\cdot, \cdot \rangle \in \overline{\RR}^{X\times Y}$ from Definition \ref{bal} is defined as the Lebesgue integral with respect to a measure or charge:
\begin{equation*}
\langle \mu, v \rangle:=\int v\dd \mu, \quad (\mu, v)\in X\times Y. 
\end{equation*}
\end{enumerate}
It is important the following. If we explicitly indicate a subset $Q'\subset Q$ such that  measures from $X'$ or functions from $Y'$ are defined only on $Q'$, then we pass from measures $\mu \in X$ and functions $v\in Y$ to their {\it restrictions} $\mu\bigm|_{Q'}$ and $v\bigm|_{Q'}$ to $Q'$. So, if $u\in \sbh(Q)$ and $M\in \sbh (Q)$, but $Q'\subset Q$ and  $\mathcal M(Q')\subset \Meas (Q')$, then the relation $u\curlyeqprec_{\mathcal M(Q')} M$ means that {\it there is a constant $C\in \RR$ such that}  
\begin{equation}\label{MQ}
\int_{Q'}u\dd \mu \leqslant \int_{Q'} M \dd \mu+C\quad\text{\it for each $\mu\in \mathcal M(Q')$}. 
\end{equation}
Similarly, if  $\vartheta  \in \Meas(Q)$ and $\mu \in \Meas(Q)$, but $Q'\subset Q$ and 
$\mathcal S(Q')\subset \sbh (Q')$, then the relation $\vartheta\curlyeqprec_{\mathcal S(Q')}\mu $ means that {\it there is a constant $C\in \RR$ such that}  
\begin{equation}\label{SQ}
\int_{Q'}v\dd \vartheta \leqslant \int_{Q'} v \dd \mu+C\quad\text{\it for each $v\in
 \mathcal S(Q')$}. 
\end{equation}

\begin{example}[{\rm \cite[3]{Gamelin}, \cite{Koosis}, \cite{Anderson}, \cite{Koosis96}, \cite{C-R}, \cite{C-RJ}, \cite{Schachermayer}, \cite{HN}, \cite{Kha91}, \cite{Kha93}, \cite{Kha01II}, \cite{Kha03}, \cite{Kh010}, \cite{KhaKha19}, \cite{KhaTalKha15}, \cite{BaiTalKha16}, \cite{MOS}}]\label{sbhJ} 
We denote by $\delta_o\in \Meas_{\comp}^{1+}(O)$  the {\it Dirac measure\/} with $\supp \delta_o=\{o\}$. 
If $\mu\in \Meas_{\comp}^+(O)$  is  a linear $\sbh(O)$-balayage of  $\delta_{o}$, then $\mu$ is called a {\it Jensen measure for\/} $o$. The convex set of all these Jensen  measures is denoted by $J_o(O)\subset \Meas_{\comp}^{1+}(O)$. If $D\Subset O$ is a domain with  non-polar boundary $\partial D$, $o\in D$, and 
a function $\omega_D\colon D\times \mathcal B(\partial D)\to [0,1]$ is the \textit{harmonic measure for\/} $D$ \cite[4.3]{R}, \cite{HK}, then the \textit{harmonic measure $\omega_D(o,\cdot)$ at\/} $o$ belongs to $J_o(D)\cap \Meas^{1+}(\partial D)$. 
If $b_d$ is the volume of the {\it unit ball\/} $\BB\subset \RR^d$ and  
$s_{d-1}$ is the area of the {\it unit sphere\/} $\partial \BB\subset \RR^d$, $s_0:=2$, then the restriction of  $\frac{1}{b_d}\lambda_d$ to $\BB$ and the measure  
$\frac{1}{s_{d-1}}\sigma_{d-1}$, where $\sigma_{d-1}$ is the usual measure of area on the {\it unit sphere\/}  $\partial \BB$, belong to  $J_0(r\BB)$ if  $r>1$. 
\end{example}
\begin{example}[{\rm \cite[3]{Gamelin}, \cite{Sarason}, \cite{Kha96}, \cite{Kha03}, \cite{Kha01II}, \cite{Kha07}, \cite{KudKha09}}]\label{sbhAS} If $\mu\in \Meas_{\comp}^+(O)$ is a linear  $\har(O)$-balayage of $\delta_{o}$, then  $\mu$ is called an {\it Arens\,--\,Singer  measure for\/} $o\in O$. The convex set of all these Arens\,--\,Singer measures is denoted by $AS_o(O)\supset J_o(O)$. Arens\,--\,Singer measures are often referred to as representing measures.
\end{example}

\section{Main result for the subharmonic version}\label{Sssh}  
\setcounter{equation}{0}

Given  $f\in F\subset {\overline{\RR}}^X $, we set $f^+\colon x\mapsto \max \{0,f(x)\}$ for each $x\in \dom f$, and  $F^+:=\{f\in F \colon f=f^+\}$, 
$F^{\uparrow}:=\Bigl\{ f\in {\overline{\RR}}^X \colon \text{\it there is an increasing 
 sequence $(f_j)_{j\in \NN}$, $f_j\in F$,}\\ \text{$\dom f_j=X$,
\it such that $f(x)=\lim\limits_{j\to \infty} f_j(x)$ for each $x\in X$\/ {\rm (we write $f_j\underset{j\to \infty}{\nearrow} f)$}} \Bigr\}$.
\begin{proposition}\label{pr:up}
Let $F\subset \overline\RR^X$ be closed relative to the $\max$-operation. If 
$f_{kj}\in F$, $k\in \NN$, $j\in \NN$, $\dom f_{kj}=X$, and 
$f_{kj}\underset{j\to \infty}{\nearrow} f_k\underset{k\to \infty}{\nearrow} f$,  then 
$F\ni  \max\limits_{k,j\leqslant n} f_{kj}\underset{n\to \infty}{\nearrow}f$, $(F^\uparrow)^{\uparrow}=F^\uparrow$.
\end{proposition}
\underline{Everywhere below,} if some  proposition is not proved or not commented, then this proposition is obvious and formulated only for the convenience of references to it.

\underline{Throughout this article,}  $D$ is a {\it non-empty domain in\/ $\mathbb \RR^d$}, and a {\it point} $o\in D$ or, in more general cases, a {\it non-empty subset\/} $S_o\Subset D$, $S_o\in \mathcal B(D)$,  will play a role  of an  {\it origin for} $D$.  
The theory of distributions or generalized functions uses test finite positive functions to define the usual order relation on distributions or measures/charges. We consider  various classes of test functions generated by subharmonic functions near the boundary $\partial D$ of $D$ or on $D\!\setminus\!S_o$ to study the order relation $\prec_H$ on subharmonic functions \cite{KhaTalKha15}, \cite{KhaTam17A}, \cite{KhaTam17L}, \cite[2.1]{KhaRoz18}, \cite{KhaAbdRoz18}, \cite{MenKha19}, \cite[Theorem 1]{KhaKha19}.   

So, we use the class  $\mathcal G_o(D\!\setminus\!S_o)\subset JP_o(D)$ of all extended Green's functions $g_{D'}(\cdot,o)$ together with the class  $\varOmega_o (D\!\setminus\!S_o)\subset J_o(D)$ of all harmonic measures  $\omega_{D'}(o, \cdot)$, whose domains $D' \Subset \RR^d$ run through all regular (for the Dirichlet problem) domains such that $S_o\Subset D'\Subset D$.
We also use various wider classes of test subharmonic functions on $D\!\setminus\!S_o$ and their extensions.
We define  subclasses $\sbh_*(D\!\setminus\!S_o)$ of functions that {\it vanish on  the boundary $\partial D\subset \RR_{\infty}^d$:}
\begin{equation}\label{{s0}0}
\sbh_0(D\!\setminus\! S_o):=\Bigl\{ v\in \sbh (D\!\setminus\!S_o)
\colon \lim_{D\ni x'\to x} v(x')=0 \text{ \it for each $x\in \partial D$}\Bigr\},
\end{equation}
and {\it vanish near the boundary $\partial D$:}
\begin{equation}\label{{s0}00}
\sbh_{00}(D\!\setminus\!S_o):=\bigl\{ v\overset{\eqref{{s0}0}}{\in} \sbh (D\!\setminus\!S_o)
\colon   v\equiv 0 \text{ \it on $D\!\setminus\!S(v)$ for some
$S(v)\Subset D$}\bigr\}.
\end{equation}
Under notations 
\begin{equation}\label{Ob}
\begin{split}
\sbh^+(Q)&:=\bigl(\sbh(Q)\bigr)^+, \quad Q\subset \RR^d,\\
\sbh (Q; \leq b_+)&:=
\bigl\{v\in \sbh(Q)\colon v\leq b_+\text{ on $Q$}\bigr\},  
\quad b_+ \in \RR, 
\end{split}
\end{equation}
using \eqref{{s0}0}--\eqref{{s0}00}, we define classes of {\it test  
 functions}
\begin{equation}\label{sbh0}
\begin{split}
\sbh_0(D\!\setminus\!S_o; \leq b_+)
&:=\sbh_0(D\!\setminus\! S_o)\bigcap \sbh (D\!\setminus\!S_o; \leq b_+),\\
\sbh_0^+(D\!\setminus\!S_o; \leq b_+)
&:=\Bigl(\sbh_0(D\!\setminus\!S_o; \leq b_+)\Bigr)^+,\\
\sbh_0^{+\uparrow}(D\!\setminus\!S_o; \leq b_+)
&:=\Bigl(\sbh_0^+(D\!\setminus\!S_o; \leq b_+)\Bigr)^{\uparrow},
\\
\sbh_{00}(D\!\setminus\!S_o; \leq b_+) &:=\sbh_{00}(D\!\setminus\!S_o)\bigcap \sbh(D\!\setminus\!S_o; \leq b_+)\\
\sbh_{00}^+(D\!\setminus\!S_o; \leq b_+) &:=\Bigl(\sbh_{00}(D\!\setminus\!S_o; \leq b_+) \Bigr)^+.
\end{split}
\end{equation}

A single-point set $\{o\}$ is denoted as $o$. For example, 
$O\!\setminus\!o:=O\!\setminus\!\{o\}$, $\RR\!\setminus\!0:=\RR\!\setminus\!\{0\}$,  
$o\cup Q:=\{o\}\cup Q$ for $Q\subset \RR_{\infty}^d$, etc. 

\begin{definition}[{\rm \cite{R}, \cite{HK}, \cite{Landkoff}}]\label{df:kK} 
 For $q\in \RR$ and $d\in \NN$, we set  
\begin{subequations}\label{kK}
\begin{align}
k_q(t)& := \begin{cases}
\ln t  &\text{ if $q=0$},\\
 -\sgn (q)  t^{-q} &\text{ if $q\in \RR\!\setminus\!0$,} 
\end{cases}
\qquad  t\in \RR^+\!\setminus\!0,
\tag{\ref{kK}k}\label{{kK}k}
\\
K_{d-2}(x,y)&:=\begin{cases}
k_{d-2}\bigl(|x-y|\bigr)  &\text{ if $x\neq y$},\\
 -\infty &\text{ if $x=y$ and $d\geqslant 2$},\\
0 &\text{ if $x=y$ and  $d=1$},\\
\end{cases}
\quad  (x,y) \in \RR^d\times \RR^d.
\tag{\ref{kK}K}\label{{kK}K}
\end{align}
\end{subequations}
 \end{definition}

\begin{example}[{\rm \cite[3]{Gamelin}, \cite{Sarason}, \cite{Kha96}, \cite{Kha01II}, \cite{Kha03}, \cite[Definition 6]{Kha07},  \cite[\S~4]{KudKha09}}]\label{expASP}
A function $V\in \sbh_{00}(O\!\setminus\!o)$  
is called an {\it Arens\,--\,Singer potential  on $O$ with pole at $o\in O$\/}  if 
\begin{equation}\label{ASpc+}
V(y)\leqslant -K_{d-2}(o,y)+O(1)\quad\text{for $o\neq y\to o$}. 
\end{equation}

The class of all Arens\,--\,Singer potentials  on $O$ with pole at $o\in O$ denote by $ASP(O\!\setminus\!o)$.  
Besides, we use a subclass 
\begin{equation}\label{ASP1}
ASP^1(O\!\setminus\!o):=\bigl\{V\in ASP(O\!\setminus\!o)\colon 
V(y)=-K_{d-2}(o,y) +O(1) \text{ for $o\neq y\to o$}
\bigr\}.
\end{equation} 

\end{example}
\begin{example}[{\rm \cite[3]{Gamelin},  \cite{Anderson}, \cite{Kha03}, 
\cite{MOS}, \cite[Definition 8]{Kha07}, \cite[IIIC]{Koosis96}, \cite{Kha12}, \cite{KhaTalKha15}, \cite{BaiTalKha16}}]\label{expJP}
The class $JP(O\!\setminus\!o):=\bigl(ASP(O\!\setminus\!o)\bigr)^+$  is the class of  {\it Jensen potentials  on $O$ with pole at\/} $o\in O$. 
Besides, we use a subclass  $JP^1(O\!\setminus\!o):=\bigl(ASP^1(O\!\setminus\!o)\bigr)^+$.
For $D\Subset O$,  
the {\it extended  Green's function  $g_D (\cdot, o)$  with pole at} 
$o\in D$ \cite[5.7.2--4]{HK}, \cite[Ch.~5, 2]{Helms},
\begin{equation}\label{g}
g_D(x,o)=\begin{cases}
g_D(x,o) &\text{ if $x\in D\!\setminus\!o$}\\ 
0 &\text{ if $x\in \RR_{\infty}^d\!\setminus\!\clos D$}\\
\limsup\limits_{D\ni x'\to x} g_D(x',o) &\text{ if $x\in \partial D$} 
\end{cases}
\quad\in \sbh^+(\RR_{\infty}^d\!\setminus\!o) ,
\end{equation}
belongs to  $JP^1(O\!\setminus\!o)$.
\end{example}

Let $Q\subset  \RR_{\infty}^d$. 
The class $C^{\infty}(Q)$  consists of the restrictions to $Q$ of  all \textit{infinitely differentiable function\/} on open sets $O\subset \RR_{\infty}^d$ that include $Q$, 
but $C(Q)$ is the vector space over $\RR$ of all continuous functions on $Q$. 
For $Q\in \mathcal B(\RR_{\infty}^d)$, the class
$C^{\infty}(Q) \dd \lambda_d$  consists of  all {\it charges $\mu \in  \Meas (Q)$ 
with  an  infinitely differentiable density,\/} i.e.,  $\dd \mu=g \dd \lambda_d$, where $g\in C^{\infty}(Q)$.  
The {\it  Riesz measure of\/} $u\in \sbh_*(O)$  is a  measure 
\begin{equation}\label{df:cm}
\varDelta_u:= c_d {\bigtriangleup}  u\in \Meas^+( O), 
\quad 
c_d:=\frac{1}{s_{d-1}\max \{1, d-2\bigr\}}, 
\end{equation}
where ${\bigtriangleup}$  is  the {\it Laplace operator\/}  acting in the sense of the
distribution theory, or the theory of generalized functions. 
If $-\infty\in \sbh ( O)$  is the {\it minus-infinity function\/} 
on $O$, then $\varDelta_{-\infty}(S):=+\infty$ for each  $S\subset 
  O$.

\begin{theorem}[{\rm subharmonic version}]\label{crit1}
Let  $M\in \sbh (D)\cap C(D)$ be a function with  Riesz measure $\varDelta_M$, let $u\in \sbh_*(D)$ be a function  with  Riesz measure $\varDelta_u$, and let  the boundary $\partial D$ be non-polar. Then  the following nine statements are equivalent:

\begin{enumerate}[{\rm [{\bf s}1]}]
\item\label{{s}1} $u\prec_{\sbh_*(D)} M$, i.e., there is $h\in \text{\rm sbh}_* (D)$ such that $u+h\leq M$ on $D$.
\item\label{{s}1J} $u\curlyeqprec_{J_o(D)}M$, i.e., $M$ is an affine $J_o(D)$-balayage of $u$\/ {\rm (see \eqref{MQ})}.
\item\label{{s}1Jinfty}  There is  a non-empty subset $S_o\Subset D$  such that  the function  $M$ is an affine $\Bigl(J_o(D)\bigcap  \Meas^{+1}_{\comp}(D\!\setminus\!S_o)\bigcap \bigl(C^{\infty}(D)\dd \lambda_d\bigr)\Bigr)$-balayage of $u$.
\item\label{{s}1Om} There is  a subset  $S_o\Subset D$  such that 
$u\curlyeqprec_{\varOmega_o(D\!\setminus\!S_o)}M$.
\item\label{{s}3G} There is a subset $S_o\Subset D$ such that
$\varDelta_u\curlyeqprec_{\mathcal G_o(D\setminus S_o)}\varDelta_M$\/ {\rm (see \eqref{SQ})}. 
\item\label{{s}1P} $\varDelta_u\curlyeqprec_{JP(D\!\setminus\!o)}\varDelta_M$, i.e., $\varDelta_M$ is an affine $JP(D\!\setminus\!o)$-balayage of $\varDelta_u$.
\item\label{{s}2} For each $S_o$  satisfying 
\begin{equation}\label{S0seto}
o \in \Int S_o\subset S_o\Subset D\subset \RR^d, \quad\text{and\/ } b_+\in \RR^+\!\setminus\!0,
 \end{equation}
$\varDelta_M $ is an affine $\sbh_0^{+\uparrow}(D\!\setminus\!S_o; \leq b_+)$-balayage of $\varDelta_u$.

\item\label{{s}3} There are  a subset $S_o\Subset D$ and  a number $b_+>0$ as in \eqref{S0seto} such that   $\varDelta_M $ is an affine $\Bigl(\sbh_{00}^+(D\!\setminus\!S_o;\leq b_+)\bigcap  C^{\infty}(D\!\setminus\!S_o)\Bigr)$-balayage of 
$\varDelta_u$.

\item\label{s9} There is  a non-empty subdomain $D_o\Subset D$ 
containing $o\in D_o$ such that   $\varDelta_M $ is an affine $\Bigl(JP^1(D\!\setminus\!o)\bigcap \har (D_o\!\setminus\!o) \bigcap C^{\infty} (D\!\setminus\!o)\Bigr)$-balayage of $\varDelta_u$.
\end{enumerate}
\end{theorem}
\begin{comments}[{\rm to Theorem \ref{crit1}}]\label{rem:6} 
{\rm The equivalence  [s\ref{{s}1}]$\Leftrightarrow$[s\ref{{s}1J}]  
is proved in \cite[The\-o\r-em 7.2]{Kha01II}.
The inclusion  $\varOmega_o(D\!\setminus\! S_o)\subset  J_o(D)$ gives 
the implications [s\ref{{s}1J}]$\Rightarrow$[s\ref{{s}1Om}],   
and the inclusion $J_o(D)\bigcap  \Meas^{+1}_{\comp}(D\!\setminus\!S_o)\bigcap \bigl(C^{\infty}(D)\dd \lambda_d\bigr)\subset J_o(D)$ gives the 
implication  [s\ref{{s}1J}]$\Rightarrow$[s\ref{{s}1Jinfty}].
The implication [s\ref{{s}1Om}]$\Rightarrow$[s\ref{{s}1}] 
for $d=2$ is proved partially in \cite[Main Theorem]{Kha07}
 and more generally in \cite[Theorem 4]{KhaKha19}. The combination of these proofs is transferred almost verbatim to the cases $d=1$ and $d>2$. We omit this transfer.
The  equivalence [s\ref{{s}1Om}]$\Leftrightarrow$[s\ref{{s}3G}] 
follows easily from the classical Poisson\,--\,Jensen formula 
\cite[4.5]{R}, \cite[3.7]{HK}. 
The implication [s\ref{{s}1P}]$\Rightarrow$[s\ref{{s}3G}]
is obvious, since $\mathcal G_o(D\!\setminus\!S_o)\subset JP(D\!\setminus\!o)$. 
The implication [s\ref{{s}1J}]$\Rightarrow$[s\ref{{s}1P}] is a special case of 
\cite[Theorem 6]{Kha07}. 
The implication [s\ref{{s}2}]$\Rightarrow$[s\ref{{s}3}]
follows from  inclusions \eqref{incv} of Proposition \ref{pr:12} below.
Thus, if we prove the implication [s\ref{{s}1P}]$\Rightarrow$[s\ref{{s}2}] 
(see Sec. \ref{AStf})
and  the chain of implications  [s\ref{{s}3}]$\Rightarrow$[s\ref{s9}]$\Rightarrow$[s\ref{{s}1Jinfty}]$\Rightarrow$[s\ref{{s}1}] (see Sec.~\ref{ASJem}--\ref{proofs31}), then Theorem \ref{crit1} will be completely proved.}
\end{comments}

\begin{remark}\label{rems1} Only the proof of implication [s\ref{{s}3}]$\Rightarrow$[s\ref{s9}] uses that the boundary $\partial D$ is \textit{non-polar.\/}  Thus, all other implications explicitly written above in the Comments to Theorem \ref{crit1} are true \textit{for any domain }$D\subset \RR^d$.
\end{remark}

\section{Main result for the harmonic version}\label{Sshv}
\setcounter{equation}{0}

``Subharmonic'' Theorem  \ref{crit1}  has a  ``harmonic'' counterpart. 
We need some additional definitions and notations.
Denote by 
$B(x,r):=\bigl\{x'\in \RR^d \colon |x'-x|<r\bigr\}\subset \RR^d$
an \textit{open ball centered at $x\in \RR^d$ with radius $r\in \RR^+$},  
$ \BB:= B(0,1)$;
\begin{equation*}
Q^{\cup r}:=\bigcup \bigl\{B(x,r)\colon x\in Q\bigr\}
\end{equation*}
is the {\it outer $r$-parallel open set for\/} $Q\subset \RR^d$ \/ {\rm \cite[Ch.~I,\S~4]{Santalo}}. 
\begin{proposition}\label{llemmaD} 
Let a subset $S\Subset \RR^d$  be connected, and $r\in \RR^+\!\setminus\!0$. 
Then $S^{\cup r}$ is a domain. If $r'\in \RR^+$ and $r'>r$, then there is a  regular (for the Dirichlet problem) domain $D_r'\Subset \RR^d$
{\rm \cite[4.1]{R}, \cite{HK}} such that
\begin{equation}\label{SDS}
S^{\cup r}\Subset D'_r\Subset S^{\cup r'}. 
\end{equation}
\end{proposition}

We supplement the classes \eqref{Ob} of subharmonic functions from Sec.~\ref{Sssh} with subclasses that are {\it positive near the boundary $\partial D$:}
 \begin{equation}\label{s0}
\begin{split}
\sbh_+(D\!\setminus\!S_o):=&\bigl\{ v\in \sbh (D\!\setminus\!S_o)
\colon   0\leq v \text{ \it  on $D\!\setminus\!S(v)$
for some $S(v)\Subset D$}\bigr\},
\\
\sbh_{+0}(D\!\setminus\!S_o)&:=
\sbh_0(D\!\setminus\!S_o)\bigcap \sbh_{+}(D\!\setminus\!S_o)\overset{\eqref{{s0}00}}{\supset} 
\sbh_{00}(D\!\setminus\!S_o),\\
\sbh_{+0}(D\!\setminus\!S_o; \leq b_+)&:=
\sbh_{+0}(D\!\setminus\!S_o)\bigcap 
\sbh(D\!\setminus\!S_o; \leq b_+).
\end{split}
\end{equation}
We define the average of  $v\in L^1\bigl(\partial B(x,r)\bigr)$ on  a sphere $\partial B(x,r)$  as
\begin{equation*}
v^{\circ r}(x):=\frac{1}{s_{d-1}}\int_{\partial \BB} v (x+rs)\dd \sigma_{d-1}(s).
\end{equation*} 
For given constants 
\begin{equation}\label{bbpmr}
-\infty <b_-< 0 < b_+<+\infty, \quad 0<4r<\dist(S_o, \partial D),
\end{equation}
using \eqref{sbh0}, we define  the following classes of \textit{test 
 functions} with some  restrictions from above and from below:
\begin{equation*}
\begin{split}
\sbh_{+0}(D\!\setminus\!S_o; r,b_-< b_+)&:=
 \bigl\{v\in \sbh_{+0}(D\!\setminus\!S_o;\leq b_+)\colon  b_-\leq v \text{ on } S_o^{\cup (4r)}\!\setminus\!S_o\bigr\},
\\
\sbh_{+0}^{\uparrow}(D\!\setminus\!S_o;  r, b_-< b_+)
&:=\Bigl(\sbh_{+0}(D\!\setminus\!S_o;  r, b_-< b_+)\Bigr)^{\uparrow},
\\
\sbh_{00}(D\!\setminus\!S_o; r,b_-< b_+)&:=
 \bigl\{v\in \sbh_{00}(D\!\setminus\!S_o; \leq b_+)\colon  b_-\leq v \text{ on } S_o^{\cup (4r)}\!\setminus\!S_o\bigr\},
\\
\sbh_{+0}(D\!\setminus\!S_o; \circ r,b_-< b_+)&:=\bigl\{v\in \sbh_{+0}(D\!\setminus\!S_o; \leq b_+)\colon  b_-\leq v^{\circ r}
 \text{ on } S_o^{\cup (3r)}\!\setminus\!S_o^{\cup (2r)}\bigr\}\\
\sbh_{+0}^{\uparrow}(D\!\setminus\!S_o;\circ r, b_-< b_+)&:=\Bigl(\sbh_{+0}(D\!\setminus\!S_o;\circ r, b_-< b_+)\Bigr)^{\uparrow}.
\end{split}
\end{equation*}
 
\begin{proposition}\label{pr:12} We have the following  inclusions:
\begin{equation}\label{incv}
\begin{array}{ccccccc}
&{\sbh_{00}^+(D\!\setminus\!S_o,\leq b_+)}& {\subset} 
&{\sbh_{00}(D\!\setminus\!S_o; r,b_-< b_+)}&\\
&\cap& &\cap&\\
&{\sbh_{0}^+(D\!\setminus\!S_o,\leq b_+)}&\subset &{\sbh_{+0}(D\!\setminus\!S_o;  r,b_-< b_+)}&\\
&\cap& &\cap&\\
&{\sbh_{0}^{+\uparrow}(D\!\setminus\!S_o,\leq b_+)}&\subset &{\sbh_{+0}^{\uparrow}(D\!\setminus\!S_o; \circ r,b_-< b_+)}&
\end{array} 
\end{equation}
In\/ \eqref{incv},  generally speaking, all inclusions are strict.
\end{proposition}
\begin{proof} Inclusions \eqref{incv} immediately follow from definitions. 
Examples from \cite[XI B2]{Koosis}, \cite[Example]{MenKha19}  show that all ``horizontal'' inclusions are strict. The first line of  ``vertical'' inclusions is strict in an obvious way. 
The second line of ``vertical'' inclusions is strict in the case when there are irregular points on the boundary $\partial D$ of the domain $D$ \cite[Lemma 5.6]{HK},  
since the limit values of the Green's function $g_D$ at such points are not zero, even if these limit values  exist \cite[Theorem 5.19]{HK}.
\end{proof}

\begin{theorem}[{\rm harmonic version}]\label{crit2}
If the conditions of Theorem\/ {\rm \ref{crit1}} are fulfilled,  then  the following eight 
 statements are equivalent:

\begin{enumerate}[{\rm [{\bf h}1]}]
\item\label{{h}1} $u\prec_{\har(D)} M$, i.e., there is  $h\in \har (D)$ such that
 $u+h\leq M$ on $D$. 
\item\label{{h}1J} $u\curlyeqprec_{AS_o(D)}M$, i.e., $M$ is an affine $AS_o(D)$-balayage of $u$\/ {\rm (see \eqref{MQ})}.
\item\label{{h}1Jinfty}  There is  a non-empty  set $S_o\Subset D$  such that   the function $M$ is an affine $\Bigl(AS_o(D)\bigcap  \Meas^{+1}_{\comp}(D\!\setminus\!S_o)\bigcap \bigl(C^{\infty}(D)\dd \lambda_d\bigr)\Bigr)$-balayage of 
$u$.
\item\label{{h}1P} $\varDelta_u\curlyeqprec_{ASP(D\!\setminus\!o)}\varDelta_M$, i.e., $\varDelta_M$ is an affine $ASP(D\!\setminus\!o)$-balayage of $\varDelta_u$\/ {\rm (see \eqref{SQ})}.  
\item\label{{h}2+} For each  connected set $S_o$  from \eqref{S0seto} and  for any constants $r, b_{\pm}$ from\/ \eqref{bbpmr}
there is a constant\/ $C\in\RR$ such that
\begin{equation}\label{almB}
\int_{D\!\setminus\!S_o}v \dd \varDelta_u\leqslant 
\int_{D\!\setminus\!S_o^{\cup(4r)}}v \dd \varDelta_M+C \quad
\text{ for each } v\in \sbh_{+0}^{\uparrow}(D\!\setminus\!S_o;\circ r, b_-< b_+).
\end{equation}
\item\label{{h}2} For each  connected set $S_o$  from \eqref{S0seto} and  for any constants  $r, b_{\pm}$
from\/ \eqref{bbpmr},   $\varDelta_M $ is an affine $\sbh_{+0}^{\uparrow}(D\!\setminus\!S_o;  r, b_-< b_+)$-balayage of  $\varDelta_u$.

\item\label{{h}3} There are  connected set $S_o$ as in \eqref{S0seto} and constants as in\/  \eqref{bbpmr}  such that   $\varDelta_M $ is an affine $\Bigl(\sbh_{00}(D\!\setminus\!S_o; r,b_-< b_+)\bigcap C^{\infty}(D\!\setminus\!S_o)\Bigr)$-balayage of $\varDelta_u$.

\item\label{h8} There is  a non-empty subdomain $D_o\Subset D$ 
containing $o\in D_o$ such that   $\varDelta_M $ is an affine $\Bigl(ASP^1(D\!\setminus\!o)\bigcap \har (D_o\!\setminus\!o) \bigcap C^{\infty} (D\!\setminus\!o)\Bigr)$-balayage of 
$\varDelta_u$.
\end{enumerate}
\end{theorem}

\begin{comments}[{\rm to Theorem \ref{crit2}}]
{\rm The implication [h\ref{{h}1}]$\Rightarrow$[h\ref{{h}1J}] is a very special case of \cite[Proposition 7.1]{Kha07} or  \cite[Corollary 8.1-I]{KhaRozKha19}. 
The implication [h\ref{{h}1J}]$\Rightarrow$[h\ref{{h}1Jinfty}] is obvious, since 
$AS_o(D)\bigcap  \Meas^{+1}_{\comp}(D\!\setminus\!S_o)\bigcap \bigl(C^{\infty}(D)\dd \lambda_d\bigr)\subset AS_o(D)$. 
The implication [h\ref{{h}1J}]$\Rightarrow$[h\ref{{h}1P}]
is noted in \cite[Theorem 6, (3.20)$\Rightarrow$(3.22)]{Kha07}.
The implication [h\ref{{h}2}]$\Rightarrow$[h\ref{{h}3}] follows from 
$\sbh_{00}(D\!\setminus\!S_o; r,b_-< b_+)\bigcap  C^{\infty}(D\!\setminus\!S_o) \subset \sbh_{+0}^{\uparrow}(D\!\setminus\!S_o;  r, b_-< b_+)$ gives . 
Thus, if we prove the implications [h\ref{{h}1P}]$\Rightarrow$[h\ref{{h}2+}]$\Rightarrow$[h\ref{{h}2}]
 (see Sec. \ref{AStf}) and [h\ref{{h}3}]$\Rightarrow$[h\ref{h8}]$\Rightarrow$[h\ref{{h}1Jinfty}]$\Rightarrow$[h\ref{{h}1}] (see Sec.~\ref{ASJem}--\ref{proofs31}), then Theorem \ref{crit2} will be completely proved.}
\end{comments}

\begin{remark}\label{remh1} 
 Only the proof of implication[h\ref{{h}3}]$\Rightarrow$[h\ref{h8}] uses that the boundary $\partial D$ is \textit{non-polar.\/}  Thus, all other implications explicitly written above in the Comments to Theorem \ref{crit2} are true \textit{for any domain }$D$.
\end{remark}

\section{Applications to the distribution of zeros of  holomorphic functions}\label{secH}
\setcounter{equation}{0}

For $n \in \NN$ we  denote by $\mathbb C^n$ the {\it $n$-dimensional complex  space over $\CC$\/} with the standard {\it norm\/} $|z|:=\sqrt{|z_1|^2+\dots+|z_n|^2}$ for $z=(z_1,\dots ,z_n)\in \CC^n$ and the distance function $\dist (\cdot, \cdot)$. By $\CC^n_{\infty}:=\CC^n\cup \{\infty\}$, and $\CC_{\infty}:=\CC_{\infty}^1$ we denote the \textit{one-point Alexandroff compactifications of $\CC^n$, and $\CC$;\/} $|\infty|:=+\infty$. 
If necessary, we identify $\CC^n$ and $\CC^n_{\infty}$
with $\RR^{2n}$ and $\RR^{2n}_{\infty}$ respectively (over $\RR$), and the preceding terminology and concepts are  naturally transferred  from $\RR^{2n}$  to $\CC^n$ and from $\RR^{2n}_{\infty}$  to $\CC^n_{\infty}$.

We use the  \textit{outer Hausdorff $p$-measure} $\varkappa_p$ with $p\in \NN_0$
\cite[A6]{Chirka}:
\begin{subequations}\label{df:sp}
  \begin{align}
\varkappa_p(S)&:=b_p \lim_{0<r\to 0} \inf \biggl\{\sum_{j\in \NN}r_j^p\,\colon  
  S\subset \bigcup_{j\in \NN}B(x_j,r_j),  0\leq r_j<r\biggr\}, \quad S\subset \RR^d,
\tag{\ref{df:sp}H}\label{df:spH}
 \\ 
b_p &\overset{\eqref{df:cm}}{:=}
\begin{cases} 1\quad&\text{if  $p=0$,}\\
\dfrac{s_{p-1}}{p}\quad&\text{ if $p\in \NN,$}
\end{cases}  \quad \text{is the {\it volume of the unit ball $\BB$ in  $\RR^p$}}.
\tag{\ref{df:sp}b}\label{df:spb}
\end{align}
\end{subequations}
Thus, for $p=0$, for any $Q\subset \RR^d$, its Hausdorff $0$-measure $\varkappa_0(Q)$ is the {\it cardinality\/} $\#Q$ of $Q$,  $\varkappa_d=\lambda_d$
on $\mathcal B(\RR^d)$, and $\sigma_{d-1}:=\varkappa_{d-1}\bigm|_{\partial \BB}$ on $\mathcal B(\partial \BB)$. 

For a subset $Q\subset \CC^n$, the class $\Hol (Q)$ consists of restrictions to $Q$ of holomorphic functions $f$ on  an  open set $O_f\supset Q$; $\Hol_*(Q):=\Hol(Q)\!\setminus\!0$.

\paragraph{\bf Zeros of holomorphic functions of several variables} 
{\cite[Ch.~1, 1,2]{Chirka}, \cite[\S~11]{Ronkin}, \cite[Ch.~4]{Kha12}. The \textit{counting function\/ {\rm (or multiplicity function, or divisor)} of zeros} of function $f\in  \Hol_*(D)$ on a {\it domain} $D\subset \CC^n$ is a function
$\Zero_f\colon D\to \NN_0$ that can be defined as
\cite[1.5, Proposition 2]{Chirka} 
\begin{equation*}
\Zero_f(z):=\max\Bigl\{p\in \NN_0  \colon 
\limsup_{D\ni z'\to z}\frac{|f(z')|}{|z'-z|^p}<+\infty\Bigr\},
\quad z\in D,
\tag{\ref{nZn}Z}\label{Zerof0}
\end{equation*} 
with the \textit{support set}
$\supp \Zero_f=\bigl\{z\in D\colon f(z)=0\bigr\}$.
For $f=0\in \Hol(D)$, by definition, $\Zero_0\equiv +\infty$ on $D$. Each counting function of zeros $\Zero_f$ is associated with a \textit{ counting measure of zeros} $n_{\Zero_f} \in \Meas^+ (D)$ defined as a Radon measure:
\begin{subequations}\label{nZn}
\begin{align}
n_{\Zero_f}(c)&:= \int_D c \dd n_{\Zero_f} 
:= \int_D c  \Zero_f \dd \varkappa_{2n-2},\quad
c \in C_0(D):= \bigl\{c\in C(D)\colon \supp c \Subset D\bigr\},
\tag{\ref{nZn}R}\label{{nZn}R}
\\
\intertext{or, equivalently, as a Borel measure on $D$:}
n_{\Zero_f} (Q) &= \int_Q
\Zero_f \dd \varkappa_{2n-2} 
\quad\text{\it for each compact subset $Q \Subset D$}.
\tag{\ref{nZn}B}\label{{nZn}B}
\end{align}
\end{subequations} 

\begin{PLf}[{\rm \cite{Lelong}, \cite{Chirka}}]
If $f \in \Hol_*(D)$, i.e.,  $\ln |f| \in \sbh_* (D)$, then
\begin{equation}\label{nufZ}
\varDelta_{\ln |f|}\overset{\eqref{df:cm}}{:=}  c_{2n}\bigtriangleup\! \ln |f| \overset{\eqref{df:cm}}{=}\frac{(n-1)!}{2\pi^n \max\{1,2n-2\}}\bigtriangleup\! 
 \ln |f| {=}n_{\Zero_f}.
\end{equation}
\end{PLf}
Let $Z \colon D \to \RR^+$ be a function on $D$.
We call this function $Z$   a \textit{subdivisor of zeros\/} for function $f \in  \Hol(D)$ if $Z\leq  \Zero_f$ on $D$. Integrals with respect to a positive measure whose integrands contain a subdivisor are everywhere below treated as upper integrals $\int^*$ \cite{Bourbaki}, \cite{ConCor}.
  
\paragraph{\bf Zeros of holomorphic functions of one variable}
\cite[0.1]{Kha12}. Let $D\subset \CC$. The counting function of zeros of $f\in  \Hol_*(D)$ is the function
\begin{subequations}
\begin{align*}
\Zero_f(z)&\overset{\eqref{Zerof0}}{=}\max \Bigl\{ p\in \NN_0\colon 
\frac{f}{(\cdot-z)^p}\in \Hol(D)\Bigr\}, \quad z\in D,
\\
\intertext{and the counting measure of zeros for $f$ is defined as a Radon measure:}
n_{\Zero_f}(c)&\overset{\eqref{{nZn}R}}{=}\sum_{z\in D} \Zero_f(z)c(z)=\int_D c\Zero_f\dd \varkappa_0, \quad c \in C_0(D),
\\
\intertext{or as a Borel measure on $D$:}
n_{\Zero_f}(Q)&\overset{\eqref{{nZn}B}}{=}\sum_{z\in S} \Zero_f(z),\quad
Q \Subset D.
\end{align*}
\end{subequations}
In this case, the support set $\supp \Zero_f$ is a locally finite set of isolated points in $D$. An \textit{ indexed set} ${\sf Z}:=\{{\sf z}_{k}\}_{k=1,2,\dots}$ of points ${\sf z}_{k}\in D$
is {\it locally finite in} this domain $D$ if 
$\# \{k \colon {\sf z}_{k}\in Q\}<+\infty$ for each subset $Q\Subset D$.
The \textit{counting measure\/}  $n_{\sf Z}\in \Meas^+(D)$ of this indexed set ${\sf Z}$ is defined as
\begin{equation*}
n_{\sf Z}:=\sum_k \delta_{{\sf z}_k}, 
\text{ or, equivalently, }
n_{\sf Z}(Q):=\sum_{{\sf z}_k\in Q}1 \text { for each $Q\subset D$}.
\end{equation*}
 Let ${\sf Z}$ and ${\sf Z}'$ be a pair of  indexed locally finite sets in $D$. By definition, ${\sf Z}={\sf Z}'$ if $n_{\sf Z}=n_{{\sf Z}'}$, and
${\sf Z}'\subset {\sf Z}$ if $n_{{\sf Z}'}\leq n_{\sf Z}$. 
An indexed set ${\sf Z}$ is the {\it zero set} of $f\in \Hol_*(D)$ if $n_{\sf Z}=n_{\Zero_f}$. 
A function $f\in \Hol(D)$ \textit{vanishes on\/}  ${\sf Z}$
if ${\sf Z}\subset \Zero_f$.

\subsection{\bf Zero sets of holomorphic functions with restrictions on their growth}\label{SSHn}

The following Theorem \ref{ThHol} develops results from \cite[Main Theorem, Theorems 1--3]{KhaRoz18}, and from \cite[Theorem 1]{MenKha19}. Both  integrals on the right-hand sides of inequalities \eqref{in:HZ} and \eqref{in:HZb}, and a pair of integrals in inequality \eqref{in:HZb+} below  are, generally speaking, upper integrals in the sense N.~Bourbaki \cite{Bourbaki}, \cite{ConCor}.
We denote by $\dsbh (O):=\sbh(O)-\sbh (O)$ the class of all {\it $\delta$-subharmonic\/} functions on $O\subset \RR^d$ \cite{Arsove}, \cite[3.1]{KhaRoz18}.

\begin{theorem}\label{ThHol} Let $D\neq \varnothing$ be a domain in $\CC^n$,
 let
\begin{equation}\label{M}
M_+\in \sbh_*(D)\cap C(D), 
\quad  M_-\in \sbh_*(D), \quad M:=M_+-M_-\in \dsbh(D)
\end{equation}
are functions   with Riesz measures $\varDelta_{M_+},\varDelta_{M_-}\in \Meas^+(D)$ and Riesz  charge $\varDelta_M=\varDelta_{M_+}-\varDelta_{M_-}\in \Meas (D)$, respectively, and let $f\in \Hol_* (D)$ be a function such that
\begin{equation}\label{fMD}
|f|\leq \exp M \quad \text{on $D$}.
\end{equation}
Then
\begin{enumerate}[{\rm [{\sf Z}1]}]
\item\label{ZI} 
For any  connected subset $S_o\Subset D$ 
from\/ \eqref{S0seto} and  for any numbers $r, b_{\pm}$ from\/ \eqref{bbpmr}, i.e.,
$0<4r<\dist (S_o, \partial D)$, $-\infty<b_-<0<b_+<+\infty$, 
there is a constant $C\in \RR$  such that
\begin{equation}\label{in:HZ}
\int_{D\!\setminus\!S_o} v \Zero_f \dd \varkappa_{2n-2}
\leqslant \int_{D\!\setminus\!S_o^{\cup(4r)}} v \dd \varDelta_M
+ \int_{S_o^{\cup(4r)}\!\setminus\!S_o} (-v) \dd \varDelta_{M_-}+C \quad
\text{for each $v\in \sbh_{+0}^{\uparrow}(D\!\setminus\!S_o;\circ r, b_-< b_+)$.} 
\end{equation}
\item\label{ZII} For any  connected set $S_o$  from \eqref{S0seto} and  numbers $r, b_{\pm}$ from\/ \eqref{bbpmr}
there is a constant $C\in \RR$ 
such that
\begin{equation}\label{in:HZb}
\int_{D\!\setminus\!S_o} v \Zero_f \dd \varkappa_{2n-2}
\leqslant \int_{D\!\setminus\!S_o} v \dd \varDelta_M+C \quad
\text{for each } v\in \sbh_{+0}^{\uparrow}(D\!\setminus\! S_o;  r, b_-< b_+).
\end{equation}
 
\item\label{ZIII} For any  set $S_o$  from \eqref{S0seto}, constant $b_+\in \RR^+\!\setminus\!0$, and subdivisor 
$ Z\leq \Zero_f$,  there is a constant  
$C \in\RR$ such that
\begin{equation}\label{in:HZb+}
\int_{D\!\setminus\!S_o} v  Z \dd \varkappa_{2n-2}
\leqslant \int_{D\!\setminus\!S_o} v \dd \varDelta_M+C 
\quad \text{for each $v\in \sbh_0^{+\uparrow}(D\!\setminus\!S_o; \leq b_+)$.}
\end{equation}
\end{enumerate}
Besides, the implication   {\sf [Z\ref{ZI}]}$\Rightarrow${\sf [Z\ref{ZII}]} is true.
\end{theorem}
\begin{proof} We can rewrite \eqref{fMD} as
\begin{equation*}
\sbh_*(D)\ni u:=\ln |f|+M_-\leq M_+\overset{\eqref{M}}{\in} \sbh_*(D), 
\end{equation*}
and, by implication [h\ref{{h}1}]$\Rightarrow$[h\ref{{h}2+}] of  Theorem \ref{crit2}
together with Remark \ref{remh1},
there is a constant $C\in \RR$ such that 
\begin{equation*}
\int_{D\!\setminus\!S_o} v \dd (\varDelta_{\ln|f|}+\varDelta_{M_-})
\overset{\eqref{almB}}{\leqslant} \int_{D\!\setminus\!S_o^{\cup(4r)}} v \dd \varDelta_{M_+}+C
\end{equation*}
for each $v\in \sbh_{+0}^{\uparrow}(D\!\setminus\!S_o;\circ r, b_-< b_+)$.
Hence, by  Poincar\'e\,--\,Lelong formula \eqref{nufZ}, 
\begin{equation*}
\int_{D\!\setminus\!S_o} v \Zero_f \dd \varkappa_{2n-2}\leqslant
\int_{D\!\setminus\!S_o^{\cup(4r)}} v \dd (\varDelta_{M_+}-\varDelta_{M_-})
- \int_{S_o^{\cup(4r)}\!\setminus\!S_o} v\dd \varDelta_{M_-}
\end{equation*}
for each $v\in \sbh_{+0}^{\uparrow}(D\!\setminus\!S_o;\circ r, b_-< b_+)$,
and we obtain  the statement {\rm [{\sf Z}\ref{ZI}]} with \eqref{in:HZ}. 
Similary,  by  Poincar\'e\,--\,Lelong formula \eqref{nufZ} and by implication [h\ref{{h}1}]$\Rightarrow$[h\ref{{h}2}] 
of  Theorem \ref{crit2} together with Remark \ref{remh1},
we obtain the statement {\rm [{\sf Z}\ref{ZII}]} with  \eqref{in:HZb}. Besides, the implication {\sf [Z\ref{ZI}]}$\Rightarrow${\sf [Z\ref{ZII}]} follows from  the estimate 
\begin{equation*}
\int_{S_o^{\cup(4r)}\!\setminus\!S_o}|v|\dd |\varDelta_M|\leqslant
\max \{b_+,-b_-\} |\varDelta_M|(S_o^{\cup(4r)}\!\setminus\!S_o)\quad
\text{for each $v\in\sbh_{+0}(D\!\setminus\!S_o; r,b_-< b_+)$}.
\end{equation*}
Finally,  by  the Poincar\'e\,--\,Lelong formula, and by implication [s\ref{{s}1}]$\Rightarrow$[s\ref{{s}2}]
of  Theorem \ref{crit1} together with Remark \ref{rems1}, we get 
the statement {\rm [{\sf Z}\ref{ZIII}]} with \eqref{in:HZb+},
 since
\begin{equation*}
\int_{D\!\setminus\!S_o} v  Z \dd \varkappa_{2n-2}\leqslant \int_{D\!\setminus\!S_o} v \Zero_f \dd \varkappa_{2n-2} 
\end{equation*} 
for every\textit{ positive\/} function $v\in \sbh_0^{+\uparrow}(D\!\setminus\!S_o; \leq b)$.
\end{proof}
\begin{remark} If $n>1$ and the function $M$ 
from \eqref{fMD} is plurisubharmonic, then the scale of necessary conditions for the distribution of zeros of $f$ can be much wider than  in Theorem \ref{ThHol}. It should include other characteristics related to the Hausdorff measure of smaller dimension than $2n-2$. We plan to explore this elsewhere. In particular, the analytical and polynomial disks should play a key role for this  approach (see \cite[Ch.~3]{Krantz}, \cite{Schachermayer}, \cite{Poletsky93}, \cite{Po99},  \cite[\S~4]{KhaKha19}, etc.). 
\end{remark}

\subsection{\bf The case of a finitely connected domain $D\subset \CC$}\label{Dfd}

We denote by $\Conn Q$ the set of {\it all connected components of} $Q\subset \RR_{\infty}^d$.

\underline{Throughout this Subsec. \ref{Dfd},} the domain $D\subset \CC$ is\textit{ finitely connected in\/} $\CC_{\infty}$ with number of component $\#\Conn
 \partial D<+\infty$,  among which \textit{there is at least one component containing two different points.\/} In this case the  \textit{boundary} $\partial D$ of $D$ is  \textit{non-polar.}

\begin{lemma}[{\rm \cite[Lemma 2.1]{Kha07}}]\label{lem:har} If $h\in \har (D)$ is harmonic, then there are a real  number $c <\#\Conn \partial D-1$     
and a function $g\in \Hol (D)$ without zeros in $D$, i.e.,
with $n_{\Zero_g}=0$, such that
\begin{equation}\label{estHar}
\ln \bigl|g(z)\bigr|\leqslant h(z)+c^+\ln\bigl(1+|z|\bigr)\quad\text{for all $z\in D$.} 
\end{equation}
If\/ $\clos D\neq \CC_{\infty}$, then we can choose\/ $c:=0$. 
\end{lemma}

\begin{theorem}\label{th:3}  Let $M\in \dsbh(D)$ be a function from\/ \eqref{M} with $M_+\in C(D)$, and\/  ${\sf Z}:=\{{\sf z}_{k}\}_{k}$ be an  indexed 
locally finite  set in $D$ of points ${\sf z}_{k}\in D$.
If there are a connected set\/ $S_o$ as in \eqref{S0seto},  numbers $b_{\pm}$, $r$ as in\/ \eqref{bbpmr}, and $C\in \RR$ such that    
\begin{equation}\label{in:HZb+Z}
\sum_{{\sf z}_k\in D\!\setminus\!S_o} 
v ({\sf z}_k)\leqslant \int_{D\!\setminus\!S_o} v \dd \varDelta_M+C \quad
\text{for each  } v\in \sbh_{00}(D\!\setminus\!S_o; r,b_-< b_+)\bigcap  C^{\infty}(D\!\setminus\!S_o),
\end{equation}
then there are a real number  $c<\#\Conn \partial D-1$ 
and a function  $f\in \Hol(D)$ with zero set 
${\sf Z}$ such that 
\begin{equation}\label{fM}
\ln \bigl|f(z)\bigr|\leqslant M(z)+c^+\ln\bigl(1+|z|\bigr) \quad \text{for each
$z\in D$}, 
\end{equation}
where  $c:=0$ if $\clos D\neq \CC_{\infty}$.
\end{theorem}
\begin{proof} We can rewrite relation \eqref{in:HZb+Z} as
\begin{equation*}
\int_{D\!\setminus\!S_o}v \dd (n_{\sf Z}+\varDelta_{M_-})\leqslant \int_{D\!\setminus\!S_o} v \dd \varDelta_{M_+}+C 
\end{equation*}
where the constant $C$ is independent of $v\in \sbh_{00}(D\!\setminus\!S_o; r,b_-< b_+)\bigcap  C^{\infty}(D\!\setminus\!S_o)$. 
By Definition \ref{bal} in the form \eqref{SQ}, the Riesz measure  $\varDelta_{M_+}$
of function $M_+$ is an affine $\bigl(\sbh_{00}(D\!\setminus\!S_o; r,b_-< b_+)\bigcap  C^{\infty}(D\!\setminus\!S_o)\bigr)$-balayage of  $n_{\sf Z}+\varDelta_{M_-}\in \Meas^+(D)$. There is a function  $u\in \sbh_*(D)$ with the  Riesz measure $n_{\sf Z}+\varDelta_{M_-}$ \cite[Theorem 1]{Arsove}. It follows from  implication [h\ref{{h}3}]$\Rightarrow$[h\ref{{h}1}] of Theorem \ref{crit2}  that there exists a function $h\in \har(D)$ such that
$u+h\leq M_+$. According to one of Weierstrass theorems, there is  a function $f_{\sf Z}\in \Hol(D)$ with the zero set  $\Zero_{f_{\sf Z}}={\sf Z}$. Hence, using  Weyl's lemma for Laplace's equation,  we have a representation  $u=\ln |f_{\sf Z}|+M_-+H$, where $H\in \har(D)$, and 
\begin{equation*}
 \ln |f_{\sf Z}|+H+h\leq M_+-M_-\overset{\eqref{M}}{=}M \quad\text{on $D$}. 
\end{equation*}
By Lemma  \ref{lem:har}, there is a function $g\in \Hol(D)$ \textit{without zeros\/} such that 
\begin{equation*}
\ln |g|\overset{\eqref{estHar}}{\leq} H+h+c^+\ln\bigl(1+|\cdot|\bigr)\quad \text{on $D$},
\end{equation*} 
where $c$ is a constant from  Lemma \ref{lem:har}.  
Hence
\begin{equation*}
\ln |f_{\sf Z}|+\ln |g|\leq M+c^+\ln\bigl(1+|\cdot|\bigr) \quad\text{on $D$.}
\end{equation*}
 If we set $f:=gf_{\sf Z}\in \Hol(D)$, then ${\sf Z}=\Zero_f$, and  we have  \eqref{fM}. 
\end{proof}
The intersection of Theorem \ref{th:3} with Theorem \ref{ThHol}, [{\sf Z\ref{ZII}}],  gives the following criterium. 
\begin{theorem}\label{crit3} 
Under the conditions of Theorem\/ {\rm \ref{th:3}},  if the domain $D$ is simply connected with $\#\partial D>1$ or\/ $\CC_{\infty}\!\setminus\!\clos D\neq \varnothing$, then the following  four assertion are equivalent:
\begin{enumerate}[{\rm [{\bf z}1]}]
\item\label{z1} There is a function $f\in \Hol(D)$ with  $\Zero_f={\sf Z}$ such that
$|f|\leq \exp M$ on $D$. 

\item\label{z2} For any  connected set $S_o$  from\/ \eqref{S0seto} and for any  numbers $r,b_{\pm}$ from\/ \eqref{bbpmr}, 
there is a constant $C\in \RR$ such that 
\begin{equation*}
\sum_{{\sf z}_k\in D\!\setminus\!S_o} v ({\sf z}_k)
\leqslant \int_{D\!\setminus\!S_o^{\cup(4r)}} v \dd \varDelta_M
+ \int_{S_o^{\cup(4r)}\!\setminus\!S_o} (-v) \dd \varDelta_{M_-}+C  
\end{equation*}
for each test function $v\in \sbh_{+0}^{\uparrow}(D\!\setminus\!S_o;\circ r, b_-< b_+)$.

\item\label{z3} For any  connected set $S_o$  from \eqref{S0seto} and for any  numbers $r,b_{\pm}$ from\/ \eqref{bbpmr}, 
there is a constant\/ $C\in \RR$ 
such that
\begin{equation}\label{in:HZb1}
\sum_{{\sf z}_k\in D\!\setminus\!S_o} v ({\sf z}_k)
\leqslant \int_{D\!\setminus\!S_o} v \dd \varDelta_M+C
\end{equation}
for each test function $v\in \sbh_{+0}^{\uparrow}(D\!\setminus\! S_o;  r, b_-< b_+)$.
\item\label{z4} There are  connected set $S_o$  as in \eqref{S0seto},  numbers $r,b_{\pm}$ as in\/ \eqref{bbpmr}, and a constant\/ $C$  such that we have 
\eqref{in:HZb1} for each
$ v\in \sbh_{00}(D\!\setminus\!S_o; r,b_-< b_+)\bigcap  C^{\infty}(D\!\setminus\!S_o)$.
\end{enumerate}
\end{theorem}
\begin{proof} The implications   
{\rm [z\ref{z1}]}$\Rightarrow${\rm [z\ref{z2}]}$\Rightarrow${\rm [z\ref{z3}]}$\Rightarrow${\rm [z\ref{z4}]}  follows from Theorem \ref{ThHol} with   implication {\sf [Z\ref{ZI}]}$\Rightarrow${\sf [Z\ref{ZII}]} and Proposition \ref{pr:12}. The implications   {\rm [z\ref{z4}]}$\Rightarrow${\rm [z\ref{z1}]}
follow from Theorem \ref{th:3}, where  $c=0$ in \eqref{fM} according to the properties of the domain $D$.
\end{proof}

\begin{remark} A special case of Theorem  \ref{crit3} was announced in  \cite[Theorem 2]{MenKha19}. 
Besides, the works \cite[Main Theorem, Theorems 1--3]{KhaRoz18}, 
\cite[Theorems 2,4,5]{KhaKha19} contain a wide range of necessary or sufficient conditions under which there exists a function $f\in \Hol_*(D)$  that \textit{vanishes on\/} ${\sf Z}$  and satisfies the inequality $|f|\leq \exp M$ on $D$.  These results do not follow directly from  Theorem \ref{crit3},  but a significant part of these conditions follows from Theorems \ref{crit1} and \ref{ThHol}, partly in stronger forms. 
\end{remark}

\section{Gluing Theorems}\label{GT}
\setcounter{equation}{0} 

\begin{gluingtheorem}[{\rm \cite[Theorem 2.4.5]{R},\cite[Corollary 2.4.4]{Klimek}}]\label{gl:th1}
Let $\mathcal O$ be an open set in $\RR^d$, 
and let $\mathcal O_0$ be a  subset of ${\mathcal O}$. If $u\in \sbh({\mathcal O})$, $u_0\in \sbh ({\mathcal O}_0)$, 
and 
\begin{equation}\label{Uu}
\limsup_{\mathcal O_0\ni x'\to x}u_0(x')\leqslant u(x) \quad\text{for each $x\in {\mathcal O}\cap \partial {\mathcal O}_0$}, 
\end{equation}
then the formula 
\begin{equation}\label{gU}
U:=\begin{cases}
\max\{u,u_0\} &\text{ on ${\mathcal O}_0$},\\
 u &\text{ on ${\mathcal O}\!\setminus\!{\mathcal O}_0$}
\end{cases}
\end{equation}
defines a subharmonic function on ${\mathcal O}$.
\end{gluingtheorem}
\begin{gluingtheorem}\label{gl:th2}
Let $O, O_0$ be a pair of open subsets in $\RR^d$,  
and  $v\in \sbh(O)$, $ v_0\in \sbh(O_0)$ be a pair of  functions such that 
\begin{subequations}\label{g01}
\begin{align}
\limsup_{\stackrel{x'\to x}{x'\in O_0\cap O}} v(x')&\leqslant 
v_0(x) \quad\text{for each $x\in O_0\cap \partial O$},
\tag{\ref{g01}$_0$}\label{g010}
\\
\limsup_{\stackrel{x'\to x}{x'\in O_0\cap O}} v_0(x')&\leqslant v(x) \quad\text{for each $x\in O\cap \partial O_0$}.
\tag{\ref{g01}$_1$}\label{g011}
\end{align}
\end{subequations}
Then the function 
\begin{equation}\label{Vv}
V:=\begin{cases}
v_0&\text{ on $O_0\!\setminus\!O$},\\
\max \{v_0,v\}\leq v_0^++v^+&\text{ on $O_0\cap O$},\\
v&\text{ on $O\!\setminus\!O_0$,}
\end{cases}
\end{equation}
is subharmonic on $O_0\cup O$.
\end{gluingtheorem}
\begin{proof} It is enough to apply the  Gluing Theorem  \ref{gl:th1} twice:
\begin{enumerate}
\item[0)] to one pair of functions  
\begin{subequations}\label{O0}
\begin{align*}
u&:=v_0\in \sbh(O_0), \quad{\mathcal O}:=O_0;
\\
u_0&:=v\bigm|_{O\cap O_0}\in \sbh(O\cap O_0) , 
\quad {\mathcal O}_0:=O\cap O_0\subset O_0, 
\end{align*}
\end{subequations} 
under condition \eqref{g010}, which corresponds to condition \eqref{Uu};
\item[1)] 
to another pair of functions 
\begin{subequations}\label{OO}
\begin{align*}
u&:=v\in \sbh(O), \quad {\mathcal O}:=O;
\\
u_0&:=v_0\bigm|_{ O_0\cap O}\in \sbh(O_0\cap O) , 
\quad {\mathcal O}_0:= O_0\cap O\subset O, 
\end{align*}
\end{subequations} 
under condition \eqref{g011}, which corresponds to condition \eqref{Uu}.
\end{enumerate}
These two glued subharmonic functions coincide at the open intersection
$O\cap O_0$, and we obtain subharmonic function $V$ on $O_0\cup O$ defined by \eqref{Vv}.
\end{proof}

\begin{gluingtheorem}[{\rm quantitative version}]\label{gl:th3}
Let $O$ and $O_0$ be a pair of open subsets in $\RR^d$, and 
$v\in \sbh(O)$ and    $g\in \sbh(O_0)$ be a pair of  functions such that 
\begin{subequations}\label{g01v}
\begin{align}
-\infty <m_v\leqslant & \inf_{x\in O\cap \partial O_0} v(x), 
\tag{\ref{g01v}m}\label{g01vOm}
\\ 
 \sup_{x\in O_0\cap \partial O}\limsup_{\stackrel{x'\to x}{x'\in O_0\cap O}} v(x')&\leqslant M_v<+\infty, 
\tag{\ref{g01v}M}\label{g01vOM}
\\
-\infty < \sup_{x\in O\cap \partial O_0}
\limsup_{\stackrel{x'\to x}{x'\in O\cap O_0}} g(x')\leqslant m_g&
< M_g\leqslant  \inf_{x\in O_0\cap \partial O} g(x) <+\infty.
\tag{\ref{g01v}g}\label{g01g}
\end{align}
\end{subequations}
If we choose 
\begin{equation}\label{v0g}
v_0:=\frac{M_v^++m_v^-}{M_g-m_g} (2g-M_g-m_g)\in \sbh(O_0), 
\end{equation} 
then the function $V$ from \eqref{Vv} is subharmonic on $O_0\cup O$. 
\end{gluingtheorem}
\begin{proof}
The function $v_0$ from definition \eqref{v0g} is subharmonic on $O_0$, since this function $v_0$  has a form $\const^+g+\const$ with constants $\const^+\in \RR^+$, $\const \in \RR$. 
In addition, by construction \eqref{v0g}, for each $x\in O_0\cap \partial O$, we obtain
\begin{multline*}
\limsup_{\stackrel{y\to x}{y\in O_0\cap O}} v(y)\overset{\eqref{g01vOm}-\eqref{g01vOM}}{\leqslant}
 M_v^++m_v^-= \frac{M_v^++m_v^-}{M_g-m_g} (2M_g-M_g-m_g)\\
 \overset{\eqref{g01g}}{\leqslant}\frac{M_v^++m_v^-}{M_g-m_g} \Bigl(2 \inf_{x\in O_0\cap \partial O} g(x)  -M_g-m_g\Bigr)
=\inf_{x\in O_0\cap \partial O} \frac{M_v^++m_v^-}{M_g-m_g} \Bigl(2  g(x)  -M_g-m_g\Bigr)\\
\overset{\eqref{v0g}}{=} \inf_{O_0\cap\partial O} v_0\leqslant 
v_0(x),  
 \quad\text{for each $x \in O_0\cap \partial O$.}
\end{multline*} 
Thus,  we have  \eqref{g010}. 
Besides, by construction \eqref{v0g}, for each $x\in O\cap \partial O_0$, we obtain
\begin{multline*}
\limsup_{\stackrel{x'\to x}{x'\in O_0\cap O}} v_0(x')
\overset{\eqref{v0g}}{\leqslant}
\frac{M_v^++m_v^-}{M_g-m_g} \biggl(2\limsup_{\stackrel{x'\to x}{x'\in O_0\cap O}} g(x')-M_g-m_g\biggr)
\\
\overset{\eqref{g01g}}{\leqslant}
\frac{M_v^++m_v^-}{M_g-m_g} (2m_g-M_g-m_g)
=-(M_v^++m_v^-)\leqslant  -m_v^-\leqslant m_v
\\
\overset{\eqref{g01vOm}}{\leqslant}
 \inf_{x\in O\cap \partial O_0} v(x)\leqslant v(x) 
 \quad\text{for each  $ x\in  O\cap \partial O_0$.}
\end{multline*} 
Thus,  we have  \eqref{g011}, and  Gluing Theorem \ref{gl:th3} follows from 
Gluing Theorem \ref{gl:th2}. 
\end{proof}

\begin{remark} 
Theorems of this section can be easily transferred to the cone of plurisubharmonic functions
\cite[Corollary 2.9.15]{Klimek}.
We sought to formulate our theorems and their proofs with the possibility of their fast transport to plurisubharmonic functions and to abstract potential theories with  more general constructions based on the theories of harmonic spaces and sheaves 
 \cite{BB66},  \cite{ConCor}, \cite{BB80}, \cite{BBC81}, \cite{BH}, \cite{LNMS}, \cite{AL},   etc.
\end{remark}

\section{Gluing with Green's Function}\label{GTg}
\setcounter{equation}{0} 

Recall that a set $E\subset \RR^d$ is called {\it polar\/} if there is $u\in \sbh_*(\RR^d)$ such that 
\begin{equation*}
\Bigl(E\subset (-\infty)_u:=\{ x\in \RR^d\colon u(x)=-\infty \} \Bigr)
\Longleftrightarrow \Bigl(\text{Cap}^* E=0\Bigr),
\end{equation*}
where the set $(-\infty)_u$ is {\it minus-infinity\/} $G_{\delta}$-set for the function 
$u$,  and  
\begin{equation*}
\text{Cap}^*(E):=
\inf_{E\subset O=\Int O}  
\sup_{\stackrel{C=\clos C\Subset O}{\mu\in \Meas^{1+}(C)}} 
 k_{d-2}^{-1}\left(\iint K_{d-2} (x,y)\dd \mu (x) \dd \mu(y) \right)
\end{equation*}
is the {\it outer capacity\/} of $E\subset \RR^d$ \cite{R}, \cite{HK}, \cite{Helms}, \cite{Doob}, \cite{Landkoff}.   

Let $\mathcal O$ be an \textit{open  proper} subset in $\RR_{\infty}^d$. 
Consider a point $o\in \RR^d$ and subsets $S_o, S\subset \RR_{\infty}^d $  such that  
\begin{equation}\label{x0S}
\RR^d \ni o \in \Int S_o\subset S_o\Subset S \subset \Int \mathcal O
=\mathcal O\subset \RR^d_{\infty}\neq \mathcal O. 
\end{equation} 
Let $D$ be a \textit{domain\/} in $\RR_{\infty}^d$ with non-polar boundary $\partial D$ such that
\begin{equation}\label{Dg}
 o\overset{\eqref{x0S}}{\in} \Int  S_o\subset S_o\Subset D \Subset S \subset \mathcal O.
\end{equation}
Such domain $D$ possesses the extended  Green's function $g_D (\cdot, o)$   
with pole at  $o\in D$ (see Example \ref{expJP}, \eqref{g})  with the following properties:  
\begin{subequations}\label{gD}
\begin{align}
g_D(\cdot , o)&\in \sbh^+ \bigl(\RR_{\infty}^d\!\setminus\!o\bigr)\subset \sbh^+\bigl(\mathcal O\!\setminus\!o\bigr) , 
\tag{\ref{gD}s}\label{gDs}\\
g_D(\cdot ,o)&= 0\text{ on $\RR_{\infty}^d\!\setminus\!\clos D \supset \mathcal O\!\setminus\!\clos D\supset \mathcal O\!\setminus\!S$}, 
\tag{\ref{gD}$_0$}\label{gD0}\\
g_D(\cdot , o)&\in \har \bigl(D\!\setminus\!o\bigr)\subset 
\har\bigl(S_o\setminus o\bigr)\subset \har\bigl(B(o,r_o)\!\setminus\!o\bigr)
\tag{\ref{gD}h}\label{gDh}\\
\intertext{for a number $r_o\in \RR^+\!\setminus\!0$, $g_D(o,o):=+\infty$,}
g_D(x,o)&\overset{\eqref{{kK}K}}{=}-K_{d-2}(x,o)+O(1) \quad\text{as $o\neq x\to o$}.
\tag{\ref{gD}o}\label{gD0a}\\
\intertext{Besides, the following strictly positive number}
0<M_g&:=\inf_{x\in \partial S_o} g_D (x,o)=
\const^+_{o, S_o, S}, 
 \tag{\ref{gD}M}\label{Mg}\\
\intertext{depends only on $o,S_o,S,D,d$, and, by the minimum principle, we have}
g_D(x,o)-M_g&\overset{\eqref{Mg}}{\geqslant} 0\quad\text{for each  $x\in S_o\!\setminus\!o$},
\tag{\ref{gD}M+}\label{Mg+}
\end{align}
\end{subequations}
where by $\const_{a_1,a_2,\dots}\in \RR$ and $\const_{a_1,a_2,\dots}^+\in \RR^+$ we denote constants and constant functions, in general, depend on $a_1,a_2,\dots$ and, unless otherwise specified, only on them, but the dependence on dimension $d$ of $\RR_{\infty}^d$, a domain $D$, and open sets $O$ or $\mathcal O$  will be not specified and not discussed.
Properties \eqref{gD} for the extended Green's function 
$g_D (\cdot, o)$\/   are well known under conditions  \eqref{x0S}--\eqref{Dg}
\cite[4.4]{R}, \cite[5.7]{HK}.

\begin{gluingtheorem}\label{gl:th4}
Under conditions  \eqref{x0S},
suppose that a function $v \in \sbh(\mathcal O\!\setminus\!S_o)$ satisfies constraints from above and from below  in the form
\begin{equation}\label{vabS}
-\infty<m_v\overset{\eqref{g01vOm}}{\leqslant} \inf_{S\!\setminus\!S_o} v\leqslant  
\sup_{S\!\setminus\! S_o} v\overset{\eqref{g01vOM}}{\leqslant}  M_v<+\infty. 
\end{equation}
Every domain $D$ with inclusions  \eqref{Dg} possesses  the extended Green's function $g_D(\cdot , o)$ with pole $o\in \Int S_o$, properties  \eqref{gD} and the  constant $M_g$ of   \eqref{Mg} such that for the function  
\begin{subequations}\label{gDV}
\begin{align}
v_o&\overset{\eqref{v0g}}{:=}\frac{M_v^++m_v^-}{M_g} \bigl(2g_D(\cdot , o) -M_g\bigr)\in 
\sbh\bigl(\RR_{\infty}^d\!\setminus\!o\bigr)\subset 
\sbh\bigl(\Int S\!\setminus\!o\bigr),
\tag{\ref{gDV}v}\label{gDVv}
\\
\intertext{we can to define a subharmonic function}
V&\overset{\eqref{Vv}}{:=}\begin{cases}
v_o&\text{ on $S_o$}\\
\sup \{v_o,v\}\leq v_o^++v^+&\text{ on $S\!\setminus\!S_o$}\\
v&\text{ on $\mathcal O\!\setminus\!S$}
\end{cases} \quad \in \sbh_*(\mathcal O\!\setminus\!o)
\tag{\ref{gDV}V}\label{gDVV}\\
\intertext{on $\mathcal O\!\setminus\!o$ satisfying the conditions}
 V&\overset{\eqref{gDh}}{\in} \har \bigl(S_o\!\setminus\!o\bigr)\subset \har\bigl(B(o,r_o)\!\setminus\!o\bigr)
\quad\text{for a number $r_o\in \RR^+\setminus o$},
\tag{\ref{gDV}h}\label{gDhV}\\
v(x)&\overset{\eqref{gDVV}}{\leqslant} V(x)\overset{\eqref{gDVv}}{\leqslant}
 M_v^++ 2\frac{M_v^++m_v^-}{M_g} g_{D} (x,o)
	 \quad\text{for each $x\in  S\!\setminus\!S_o$,}
\tag{\ref{gDV}+}\label{gDhV+}\\
0& \leqslant V(x)\overset{\eqref{Mg+}}{\leqslant}
 2\frac{M_v^++m_v^-}{M_g} g_{D} (x,o)
\quad\text{for each $x\in  S_o\!\setminus\!o$,}
\tag{\ref{gDV}$_0^+$}\label{gDhV++}\\
V(x)&\overset{\eqref{gD0a}}{=}-2\frac{M_v^++m_v^-}{M_g} K_{d-2} (x,o)+O(1) \quad\text{as $o\neq x\to o$}.
\tag{\ref{gDV}o}\label{gD0aV}
\end{align}
\end{subequations}
\end{gluingtheorem}
\begin{proof} It is enough to apply Gluing Theorem \ref{gl:th3}
 with 
\begin{equation*}
O:=\mathcal O\!\setminus\!\clos S_o, 
\quad O_0:=\Int S\!\setminus\!o, \quad 
g:=g_D(\cdot,o), \quad  m_g:=0 
\end{equation*} 
in accordance with the reference marks indicated over relationships in 
 \eqref{vabS}--\eqref{gDV}.
\end{proof}

\begin{remark}\label{MgS} The choice of  $D$ and $M_g$ in \eqref{Mg} and \eqref{gD0aV} is entirely determined by the mutual arrangement of  $S_o\Subset S$.
\end{remark}

\begin{gluingtheorem}\label{gl:th_es}
Let $\mathcal O\subset \RR^d$ be an open subset, and  $S_o\subset \RR^d$ be a connected set such that there is a point
\begin{equation}\label{S0}
 o\overset{\eqref{x0S}}{\in} \Int  S_o\subset S_o \Subset \mathcal O.
\end{equation}
Let  $r\in \RR^+$ be a  number such that
\begin{equation}\label{posr}
0<4r<\dist(S_o, \partial \mathcal O),
\end{equation}
 $D_r:=D_r'$ be a domain from Proposition\/ {\rm \ref{llemmaD}} satisfying \eqref{SDS} with $2r$ instead of $r$ and $r':=3r>2r$, $v \in \sbh_*(\mathcal O\!\setminus\!S_o)$, and $M_v\in \RR$ be a constant  such that 
\begin{subequations}\label{avv}
\begin{align}  
v&\leq M_v <+\infty \quad\text{on $S_o^{\cup (4r)}\!\setminus\!S_o$},
\tag{\ref{avv}M}\label{avvM}
\\
m_v&:=\inf \bigl\{ v^{\circ r}(x)\colon x\in S_o^{\cup (3r)}\!\setminus\!S_o^{\cup (2r)}\bigr\}.
\tag{\ref{avv}m}\label{avvm},\\
M_g&:=\inf_{x\in \partial S_o^{\cup (2r)}} g_{D_r} (x,o)=
\const^+_{o, S_o, r,{D_r}}=\const^+_{o,S_o,r}.
\tag{\ref{avv}g}\label{gDVg}
\end{align}
\end{subequations}
Then $M_g>0$,  $m_v>-\infty $, and   there is a subharmonic function $V\in \sbh_*(\mathcal{O}\!\setminus\!o) $ satisfying conditions \eqref{gDhV}--\eqref{gD0aV}, i.e., 
\begin{subequations}\label{VK}
\begin{align}
0<V&\in \har^+(S_o\!\setminus\!o) \quad\text{on $S_o\!\setminus\!o$},
\tag{\ref{VK}h}\label{gDVh}\\
V&=v \quad \text{on $\mathcal O\!\setminus\!S_o^{\cup (4r)}$},
\tag{\ref{VK}=}\label{gDV=}\\
v(x)\leqslant V(x)&\leqslant M_v^++2\frac{M_v^++m_v^-}{M_g} g_{D_r}(x,o) \quad \text{for each $x\in  S^{\cup (4r)}\!\setminus\!S_o$},
\tag{\ref{VK}+}\label{gDVleq}
\\
0< V(x)&\leqslant 2\frac{M_v^++m_v^-}{M_g} g_{D_r}(x,o) \quad \text{for each $x\in  S_o\!\setminus\!o$},
\tag{\ref{VK}$_0^+$}\label{gDVleq+}
\\
V(x)&\overset{\eqref{gD0aV}}{=}- 2\frac{M_v^++m_v^-}{M_g}K_{d-2} (x,o)+O(1) \quad\text{as\quad  $o\neq x\to o$}.
\tag{\ref{VK}o}\label{gDVo}
\end{align}
\end{subequations}

\end{gluingtheorem}
\begin{proof}
 We have $M_g>0$ for \eqref{gDVg} by \eqref{Mg}--\eqref{Mg+}, and  $m_v>-\infty$ since the function  
$v^{\circ r}$  is continuous on $\clos (S_o^{\cup (3r)}\!\setminus\!S_o^{\cup (2r)})$
 \cite[Theorem 1.14]{Helms}. 
Using the Perron\,--\,Wiener\,--\,Brelot method \cite[Ch.~4]{R}, \cite[2.6]{HK}, \cite[Ch.~8]{Helms}, \cite[Ch.~VIII, 2]{Doob}, we construct the function 
\begin{equation*}
\check{v}:=\sup\Bigl\{w\in \sbh_*(\mathcal O \!\setminus\!S_o) \colon  
w\leq v \text{ on } (\mathcal O \!\setminus\!S_o)\!\setminus\!\bigl(S_o^{\cup (4r)}\!\setminus\!\clos S_o^{\cup r}\bigr) \Bigr\}
\end{equation*}
and its upper regularization 
 \begin{equation*}
\check{v}^*(x):=\limsup_{x'\to x}\check v(x'), 
\quad x\in \Int (\mathcal O\!\setminus\!S_o). 
\end{equation*}
By this construction, the new function $\check{v}^*$ is subharmonic on  $\Int(\mathcal O\!\setminus\! S_o)$, harmonic on  $S_o^{\cup (4r)}\!\setminus\!\clos S_o^{\cup r}$
and $\check{v}^*=v$ on $\mathcal O\!\setminus\!S_o^{\cup (4r)}$.
It follows from the principle of subordination (domination) for harmonic continuations
and  the  maximum principle that 
\begin{equation}\label{estwv}
-\infty <m_v\leq \check{v}^*\quad\text{on $S_o^{\cup (3r)}\!\setminus\!S_o^{\cup (2r)}$}, 
\qquad \check{v}^*\leq M_v\quad\text{on $S_o^{\cup (4r)}\!\setminus\!\clos S_o$}.
\end{equation}
In Gluing  Theorem \ref{gl:th4}, we choose the set  $S_o^{\cup(2r)}$ as $S_o$, $S_o^{\cup (3r)}$ as  $S$, and $\check{v}^*$ as $v$. We have  \eqref{vabS}
for $\check{v}^*$ in view of \eqref{estwv}. 
Then, by construction \eqref{gDVv}--\eqref{gDVV} and conditions \eqref{gDhV}--\eqref{gD0aV}, we get series of conclusions \eqref{VK} of Gluing Theorem \ref{gl:th_es} with $S_o^{\cup(2r)}$ instead of $S_o$. But we can replace $S^{\cup(2r)}$ with $S_o$ back in estimates \eqref{gDVleq}--\eqref{gDVleq+} by virtue of condition \eqref{avvM}, as well as in \eqref{gDVh} since $S_o\subset S^{\cup (2r)}$. The possibility of replacing a constant $\const^+_{o, S_o, r,{D_r}}$ with $\const^+_{o,S_o,r}$ follows from Remark \ref{MgS}. 
\end{proof}

\section{Gluing of test functions}\label{tsf}
\setcounter{equation}{0}

\begin{proposition}[{\rm \cite[3.2.1]{KhaRoz18}}]\label{pr:0} A function $v\overset{\eqref{s0}}{\in} \sbh_{+0}(D\!\setminus\!S_o)$ continues as subharmonic function on $\RR^d_{\infty}\!\setminus\!S_o$ by the  rule
\begin{equation}\label{pr:0v}
v(x):=\begin{cases}
v(x)&\quad\text{at $x\in D\!\setminus\!S_o$},\\ 
0&\quad\text{at $x\in \RR^d_{\infty}\!\setminus\! D$}
\end{cases}
\quad \in \sbh(\RR^d_{\infty}\!\setminus\!S_o).
\end{equation}
\end{proposition}

\begin{gluingtheorem}[{\rm for test functions}]\label{glth5}
If  $D$ is a domain together with  \eqref{S0seto}, and $b_{\pm}, r$ are constants satisfying \eqref{bbpmr}, then there is a constant
\begin{equation}\label{gDVgv}
B:=2\frac{b_+-b_-}{\const^+_{o,S_o,r}}:=\const^+_{o,S_o,r,b_{\pm}}>0.
\end{equation}
such that for any test subharmonic function $v\in \sbh_{+0} (D\!\setminus\!S_o; \circ r,b_-< b_+)$ 
we can construct a subharmonic  function $V\in \sbh_*(\RR_{\infty}^d\!\setminus\!o)$ with properties  
\begin{subequations}\label{VKv}
\begin{align}
0<V&\overset{\eqref{gDVh}}{\in} \har^+(S_o\!\setminus\!o) \quad\text{on $S_o\!\setminus\!o$},
\tag{\ref{VKv}h}\label{gDVhv}\\
V&\overset{\eqref{gDV=}}{=}v \quad \text{on $D\!\setminus\!S^{\cup (4r)}$}, 
\tag{\ref{VKv}=}\label{gDV=v}
\\
v(x)\leqslant V(x)&\overset{\eqref{gDVleq}}{\leqslant} b_++ Bg_D(x,o) \quad \text{for each  $x\in  S^{\cup (4r)}\!\setminus\!S_o$},
\tag{\ref{VKv}+}\label{gDVleqv}
\\
0< V(x)&\overset{\eqref{gDVleq+}}{\leqslant} Bg_D(x,o) \quad \text{for each  $x\in  S_o\!\setminus\!o$},
\tag{\ref{VKv}$^+_0$}\label{gDVleqv+}
\\	
V(x)&\overset{\eqref{gDVo}}{=}-BK_{d-2} (x,o)+O(1) \quad\text{as $o\neq x\to o$}
\tag{\ref{VKv}o}\label{gDVov},
\\
V&\equiv 0 \quad\text{on $\RR^d_{\infty}\!\setminus\!D$.}
\tag{\ref{VKv}$_0$}\label{gDV=0}
\end{align}
\end{subequations}
Besides, for each  $v\in  \sbh_{+0}^{\uparrow}(D\!\setminus\!S_o; \circ r,b_-< b_+)$ we get  a function $V\colon  \RR_{\infty}^d\!\setminus\!o \to \overline \RR$  with the same properties  \eqref{gDVhv}--\eqref{gDVov}
as the limit of an increasing sequence of functions satisfying the conditions \eqref{gDVhv}--\eqref{gDVov}, but with a weaker property instead of \eqref{gDV=0}, more precisely 
\begin{equation}\label{vuparr}
V\equiv 0 \quad\text{on $\RR^d_{\infty}\!\setminus\!\clos D$, \quad   $V\geq 0$ on $\partial D$,}
\tag{\ref{VKv}$'_0$}
\end{equation}
and such function $V$ is not necessarily upper semi-continuous on 
$(\clos D)\!\setminus\!S_o$.

This is also true for every positive function
\begin{equation*}
v\in \sbh_0^+(D\!\setminus\!S_o; \leq b_+)\quad\text{or}\quad  v\in \sbh_0^{+\uparrow}(D\!\setminus\!S_o; \leq b_+),
\end{equation*} 
together with an additional property of the positivity of $V\geq 0$ on $\RR_{\infty}^d\!\setminus\!o$.
\end{gluingtheorem}
\begin{proof}
By Proposition \ref{pr:0} we can consider the function $v \in \sbh_{+0} (D\!\setminus\!S_o; \circ r,b_-< b_+)$ as defined on $\RR^d_{\infty}\!\setminus\!S_o$ by \eqref{pr:0v}, i.e., $v\equiv 0$ on $\RR^d_{\infty}\!\setminus\!D$, and $v\in \sbh(\RR^d_{\infty} \!\setminus\!S_o)$. By Gluing Theorem \ref{gl:th_es}  with open set $\mathcal O:=\RR^d\!\setminus\!o$, with constants $M_v\overset{\eqref{avvM}}{:=}b^+$,
$m_v\overset{\eqref{avvm}}{:=}b_-$, $M_g\overset{\eqref{gDVg}}{:=}\const^+_{o,S_o,r}>0$,
and a constant  $B$ from \eqref{gDVgv},
we construct  a function $V\in \sbh_0(\RR^d_{\infty} \!\setminus\!o)$, $V(\infty):=0$,	
with properties \eqref{VK} that go into properties  \eqref{gDVhv}--\eqref{gDVov} together with identity \eqref{gDV=0}. 
Note that we use the principle of domination in \eqref{gDVleqv}--\eqref{gDVleqv+}
for Green's functions to replace $D_r$ with $D$,  
since a domain $D_r$ from \eqref{gDVleq} 	is a subdomain of $D$ provided \eqref{bbpmr}.
\end{proof}

\section{Approximation of test functions by Arens\,--\,Singer and Jensen potentials} 
 \setcounter{equation}{0}

\begin{theorem}\label{pr:13} Let  $b_{\pm}, r$ are constants from  \eqref{bbpmr}, $v\in \sbh_{+0}^{\uparrow}(D\!\setminus\!S_o; \circ r,b_-< b_+)$ (resp. $v\in \sbh_0^{+\uparrow}(D\!\setminus\!S_o; \leq b_+)$).  Then there are a constant  
\begin{equation}\label{BV}
B\overset{\eqref{gDVgv}}{=}\const_{o,S_o,r,b_{\pm}}^+,
\end{equation}
and  an increasing sequence $(V_n)_{n\in \NN}$ of Arens\,--\,Singer (respectively, Jensen) potentials $V_n\in ASP^1(D\!\setminus\!o)$ (respectively, $V_n\in JP^1(D\!\setminus\!o)$), $ n\in \NN$, such that 
\begin{subequations}\label{VPAS}
\begin{align}
0<&V_n\in \har(\Int S_o\!\setminus\!o)\text{ on $S_o\!\setminus\!0$},
\tag{\ref{VPAS}h}\label{{VPAS}h}
\\
BV_n&\underset{n\to \infty}{\nearrow}V \quad\text{on $D\!\setminus\!o$},
\tag{\ref{VPAS}$\uparrow$}\label{{VPAS}up}
\\
\intertext{where $V\colon  \RR_{\infty}^d\!\setminus\!o \to \overline\RR$  is a function with properties \eqref{gDVhv}--\eqref{gDVov}, \eqref{vuparr},}
V_n(x)&=-K_{d-2}(x,o)+O(1) \quad\text{as $o\neq x\to o$},
\tag{\ref{VPAS}o}\label{{VPAS}upo}
\\
BV_n&\leq  b_++Bg_D(\cdot ,o)\quad\text{on $S_o^{\cup (4r)}\!\setminus\!o$}.
\tag{\ref{VPAS}+}\label{{VPAS}b}
\end{align}
\end{subequations}
\end{theorem}
\begin{proof} The classes $ASP(D\!\setminus\!o)$ and $JP(D\!\setminus\!o)$ are closed relative to the  $\max$-operation. By Propositions \ref{pr:up} and \ref{pr:0}, it suffices to prove  Theorem \ref{pr:13} only for  functions  
\begin{equation*}
v\in \sbh_{+0}(D\!\setminus\!S_o; \circ r,b_-< b_+)\quad \text{(resp. $v\in \sbh_0^{+}(D\!\setminus\!S_o; \leq b_+)$);  $v\equiv 0$ on $\RR_{\infty}^d\!\setminus\!D$}.  
\end{equation*}
Denote by  $\conn (Q,x) \in \Conn Q$ a {\it connected component of $Q\subset \RR_{\infty}^d$ containing}  $x$.
For a function $v\in \sbh_{+0}(D\!\setminus\!S_o; \circ r,b_-< b_+)$, we  consider a function $V$ from Gluing Theorem \ref{glth5} with properties \eqref{VKv}. For each number $n\in \NN$ we put in correspondence an open set  
$ O_n:=\bigl\{x\in \RR_{\infty}^d\!\setminus\!o\colon V(x)<1/n\bigr\}\supset O_{n+1}$, and  a function $v_n$ such that 
\begin{enumerate}[{i)}]
\item\label{iv} this function $v_n$ vanishes  on all connected components $\conn (O_n,x)\in \Conn O_n$ that met with complement  $\RR_{\infty}^d\!\setminus\!D$ of $D$, i.e.,  $v_n\equiv 0$ on every  connected  component $\conn (O_n,x)$ with $x\in \RR_{\infty}^d\!\setminus\!D$,
\item\label{iiv}  $v_n:=V-1/n$ on the rest of $\RR_{\infty}^d\!\setminus\!o$.
\end{enumerate}
By construction \ref{iv})--\ref{iiv}), these functions $v_n$  are subharmonic on $\RR^d_{\infty}\!\setminus\!0$ with $\supp v_n\Subset D$. Therefore,  $v_n \overset{\eqref{{s0}00}}{\in}\sbh_{00} (D\!\setminus\!o)$. 
Besides,  these functions $v_n$ form an increasing sequence $v_n\underset{n\to \infty}{\nearrow} V$ on $\RR^d_{\infty}\!\setminus\!o$. In view of \eqref{gDVov}, there exists the limit
\begin{equation*}
v_n(x)=-K_{d-2}(x,o)+O(1) \quad\text{as $o\neq x\to o$}.
\end{equation*}
Thus, if we set $V_n:=\frac{1}{B}\,v_n$, 
then we have \eqref{ASpc+}. Hence every function $V_n$ belongs to $ASP^1(D\!\setminus\!o)$ by Example \ref{expASP}, \eqref{ASP1}, and  properties \eqref{VPAS} are fulfilled. 

In the case $v\in \sbh_0^{+}(D\!\setminus\!S_o; \leq b_+)$, we consider the functions $v_n^+$ after \ref{iv})--\ref{iiv}) instead of $v_n$ and we obtain 
$V_n\in JP^1(D\!\setminus\!o)$  with  properties \eqref{VPAS}
by Example \ref{expJP}. 
\end{proof}

\section{From Arens\,--\,Singer and Jensen potentials to test functions}\label{AStf}
\setcounter{equation}{0}

\begin{proof}[of implications 
{\rm [s\ref{{s}1P}]$\Rightarrow$[s\ref{{s}2}]} 
and\/ {\rm [h\ref{{h}1P}]$\Rightarrow$[h\ref{{h}2+}]}
of Theorems\/ {\rm \ref{crit1}, \ref{crit2}}]
By Definition \ref{bal} in variant \eqref{SQ}, the condition $\varDelta_u\curlyeqprec_{JP(D\!\setminus\!o)}\varDelta_M$  from  [s\ref{{s}1P}] (resp. $\varDelta_u\curlyeqprec_{ASP(D\!\setminus\!o)}\varDelta_M$ from [h\ref{{h}1P}])  means that there is a constant $C_1\in \RR$ such that
\begin{equation}\label{SQd}
\int_{D\!\setminus\!o}v\dd \varDelta_u \leqslant \int_{D\!\setminus\!o} v \dd \varDelta_M+C_1
\quad\text{for each $v\overset{\text{or}}{\in} \left[\begin{array}{c}
 JP(D\!\setminus\!o)\\ 
ASP(D\!\setminus\!o)
\end{array}\right.$, respectively}. 
\end{equation}

Let $v\in \sbh_0^{+\uparrow}(D\!\setminus\!S_o; \leq b_+)$ in the case [s\ref{{s}1P}]
or $v\in \sbh_{+0}^{\uparrow}(D\!\setminus\!S_o; \circ r, b_-< b_+)$ in the case [h\ref{{h}1P}], respectively. By Theorem \ref{pr:13},  there are  a constant
from \eqref{BV} and  an increasing sequence of  Jensen (respectively, Arens\,--\,Singer) potentials $V_n\in JP(D\!\setminus\!o)$ (respectively, $V_n\in ASP(D\!\setminus\!o)$), $n\in \NN$, satisfying \eqref{VPAS}.
Hence  \eqref{SQd} entails 
\begin{equation}\label{estVVn}
\int_{D\!\setminus\!o} BV_n \dd \varDelta_u\leqslant 
\int_{D\!\setminus\!o} BV_n \dd \varDelta_M+BC_1 \leqslant 
\int_{D\!\setminus\!o} V \dd \varDelta_M+BC_1, 
\end{equation}
where the function $ V:=\lim_{n\to \infty} BV_n$  on $D\!\setminus\!o$  has all the properties  \eqref{gDVhv}--\eqref{gDVov}, \eqref{vuparr}, and  the  constant $BC_1\in \RR$ independent of $V_n$, $n\in \NN$.
Let's represent the integral on the right-hand side of inequalities \eqref{estVVn} as a sum of integrals:
\begin{multline}\label{VmuM}
\int_{D\!\setminus\!o} V \dd \varDelta_M=\left(\int_{D\!\setminus\!S_o^{\cup(4r)}} +
\int_{S_o^{\cup (4r)}\!\setminus\!S_o}+
\int_{S_o\!\setminus\!o} \right)V \dd \varDelta_M\\
\overset{\eqref{gDV=v}\text{--}\eqref{gDVleqv+}}{\leqslant} \int_{D\!\setminus\!S_o^{\cup(4r)}}v \dd \varDelta_M +
b_+\varDelta_M(S_o^{\cup (4r)}\!\setminus\!S_o)+
B\int_{S_o^{\cup (4r)}\!\setminus\!o} g_D(x,o)\dd \varDelta_M
\leqslant \int_{D\!\setminus\!S_o^{\cup(4r)}}v \dd \varDelta_M
+C_2,
\end{multline}
where $C_2=\const_{o,S_o,r,b_{\pm},B,u,M}^+$
is a constant independent of $v$. In addition, in the case 
$v\in \sbh_0^{+\uparrow}(D\!\setminus\!S_o; \leq b_+)$, the function  $v$ is positive on $D\!\setminus\!S_0$, and we have 
\begin{equation}\label{+M}
\int_{D\!\setminus\!o} V \dd \varDelta_M\overset{\eqref{VmuM}}{\leqslant} \int_{D\!\setminus\!S_o}v \dd \varDelta_M
+C_2\quad\text{in the case  [s\ref{{s}1P}]}.
\end{equation}
If the integrals in the right-hand sides of \eqref{VmuM} and \eqref{+M} 
are equal to $+\infty$, then there is nothing to prove. Otherwise, by 
Beppo Levi's monotone convergence theorem for Lebesgue integral, 
\eqref{estVVn} and \eqref{{VPAS}up} together with \eqref{VmuM} and \eqref{+M}  entails
\begin{equation}\label{Levlim}
\int_{D\!\setminus\!o} V \dd \varDelta_u\leqslant 
\int_{D\!\setminus\!S_o^*} v \dd \varDelta_M+C_2,\quad\text{where 
$S_o^*\overset{\text{or}}{=} \left[\begin{array}{c}
S_o\\ 
S_o^{\cup (4r)}
\end{array}\right.$, respectively.} 
\end{equation}
According to 
\eqref{gDVhv}--\eqref{gDVleqv}, it is follows from  \eqref{Levlim}  that
\begin{multline*}
\int_{D\!\setminus\!S_o}v\dd \varDelta_u
\overset{\eqref{gDVhv}}{\leqslant}
\int_{S_0\!\setminus\!o} V\dd  \varDelta_u
+\int_{D\!\setminus\!S_0} v\dd  \varDelta_u\\
\overset{\eqref{gDVleqv}}{\leqslant}
\int_{S_0\!\setminus\!o} V\dd  \varDelta_u
+\int_{S_o^{\cup (4r)}\!\setminus\!S_0} V\dd  \varDelta_u
+\int_{D\!\setminus\!S_o^{\cup (4r)}} v\dd  \varDelta_u
\overset{\eqref{gDV=v}}{=}
\int_{D\!\setminus\!o} V\dd  \varDelta_u
\overset{\eqref{Levlim}}{\leqslant}
\int_{D\!\setminus\!S_o^*} v \dd \mu_M+C,
\end{multline*}
where the constant $C$ is independent of $v$, and  $S_o^*$ is defined in \eqref{Levlim}.
\end{proof}
\begin{proof}[of implication\/ {\rm [h\ref{{h}2+}]
$\Rightarrow$[h\ref{{h}2}]}
of Theorem\/ {\rm \ref{crit2}}]
 If  $v\in \sbh_{+0}^{\uparrow}(D\!\setminus\!S_o;  r, b_-< b_+)\subset \sbh_{+0}^{\uparrow}(D\!\setminus\!S_o;  \circ r, b_-< b_+)$,  then this function $v$ is bounded from  below on $S_o^{\cup (4r)}\!\setminus\!S_o$  by the constant $b_-\in (-\RR^+)\!\setminus\!0$, and \eqref{almB} 
implies
\begin{multline*}
\int_{D\!\setminus\!S_o} v \dd \varDelta_u\overset{\eqref{VmuM}}{\leqslant} \int_{D\!\setminus\!S_o}v \dd \varDelta_M
-\int_{S_o^{\cup (4r)}\!\setminus\!S_o}v \dd \varDelta_M+C
\\
\leqslant \int_{D\!\setminus\!S_o}v \dd \varDelta_M-\varDelta_M(S_o^{\cup (4r)}\!\setminus\!S_o)b_-+C \quad\text{for each $v\in \sbh_{+0}^{\uparrow}(D\!\setminus\!S_o;  r, b_-< b_+)$}, 
\end{multline*} 
where the constant $\bigl(-\varDelta_M(S_o^{\cup (4r)}\!\setminus\!S_o)b_-+C\bigr)$
is independent of $v$.  
\end{proof}

\section{ Arens\,--\,Singer and Jensen measures and their potentials}\label{ASP} 
\setcounter{equation}{0}

\begin{definition}[{\rm \cite{R}, \cite[Definition 2]{Kha03}, \cite[3.1, 3.2]{KhaRoz18}, \cite{Chi18}}]\label{df:pot} 
Let ${\mu} \in \Meas_{\comp}(\RR^d)$ be a charge with compact support. Its {\it potential\/}  is a  function 
\begin{equation}\label{{pmu}p} 
\pt_{\mu}\colon \RR^d\to \overline \RR, \quad \pt_{\mu}(y)\overset{\eqref{{kK}K}}{:=}\int K_{d-2}(x,y) \dd \mu (x), 
\end{equation}
where the kernel  $K_{d-2}$ is defined in Definition \ref{df:kK}
by the function $k_q$ from \eqref{{kK}k}, 
\begin{equation*}
\dom {\pt}_{\mu}=
\left\{y\in \RR^d\colon \int_{0}\frac{\mu^-(y,t)}{t^{m-1}} \dd t<+ \infty 
\right\}\bigcup
\left\{y\in \RR^d\colon \int_{0}\frac{\mu^+(y,t)}{t^{m-1}} \dd t<+ \infty 
\right\},
\end{equation*}
and $\RR^d\!\setminus\!\dom {\pt}_\mu$ is a  bounded {\it polar set\/} in $\RR^d$ with $\text{Cap}^* (\RR^d\!\setminus\!\dom {\pt}_\mu)=0$.
 \end{definition}

\begin{proposition}\label{pt_below} If 
\begin{equation}\label{muLo}
\mu \in \Meas_{\comp}^+(\RR^d),\quad
L\Subset \RR^d, \quad L\!\setminus\!o\neq \varnothing,
\end{equation} 
then 
\begin{subequations}\label{pmul}
\begin{align}
\inf_{x\in L} {\pt}_{\mu}(x)&\geqslant \mu(\RR^d)k_{d-2}\bigl(\dist(L,\supp \mu)\bigr),
\tag{\ref{pmul}i}\label{{pmul}i}
\\
\inf_{x\in L} {\pt}_{\mu-\delta_o}(x)&\geqslant \mu(\RR^d)k_{d-2}\bigl(\dist(L,\supp \mu)\bigr)-
k_{d-2}\left(\sup_{x\in L}|x|+|o|\right).
\tag{\ref{pmul}o}\label{{pmul}o}
\end{align}
\end{subequations}
\end{proposition}
\begin{proof} The case $d=1$ is trivial. Let $d\geqslant 2$.
   If $\dist(L,\supp \mu)=0$, 
then the right-hand sides in inequalities \eqref{pmul}
are equal to $-\infty$, and inequalities \eqref{pmul} are true.  Otherwise, by Definition \ref{df:pot}, we obtain
\begin{multline}\label{est:pinf}
 {\pt}_{\mu}(x)=\int k_{d-2}\bigl(|x-y|\bigr)\dd \mu(y)\geqslant
\inf_{y\in \supp \mu} k_{d-2}\bigl(|x-y|\bigr)
\mu (\RR^d) \\
\geqslant k_{d-2}\left(\inf_{y\in \supp \mu} |x-y|\right) \mu (\RR^d)=
\mu(\RR^d)k_{d-2}\bigl(\dist(x,\supp \mu)\bigr),
\end{multline}
since the function $k_q$ from \eqref{{kK}k} is \textit{increasing.\/}
We obtain inequality \eqref{{pmul}i}
 after applying the operation $\inf_{x\in L}$ to both sides of inequality \eqref{est:pinf}. Using \eqref{{pmul}i}, we have
\begin{multline*}
\inf_{x\in L} {\pt}_{\mu-\delta_o}(x) \overset{\eqref{{pmu}p}}{=}
\inf_{x\in L} \bigl({\pt}_{\mu}(x)-k_{d-2}
\bigl(|x-o|\bigr)\bigr)\geqslant 
\inf_{x\in L} {\pt}_{\mu}(x)- \sup_{x\in L}k_{d-2}
\bigl(|x-o|\bigr)
\\
\overset{\eqref{{pmul}i}}{\geqslant}\mu(\RR^d)k_{d-2}\bigl(\dist(L,\supp \mu)\bigr)-
k_{d-2}\left(\sup_{x\in L}|x|+|o|\right),
\end{multline*}
and it  is  inequality \eqref{{pmul}o}. 
\end{proof}

\begin{dualtheo}[{\rm \cite[Proposition 1.4, Duality Theorem]{Kha03}}] The mapping 
\begin{equation}\label{mcP}
\mathcal P_o \colon \mu \longmapsto {\pt}_{\mu-\delta_o}
\end{equation}
is an affine bijection from $AS_o(O)$ onto $PAS(O\!\setminus\!o)$
(resp. $J_o(O)$ onto $JP(O\!\setminus\!o)$)
 with inverse mapping
\begin{equation}\label{P-1}
\mathcal P_o^{-1} \colon V
\overset{\eqref{df:cm}}{\longmapsto} 
c_d {\bigtriangleup}V \bigm|_{\RR^d\!\setminus\!o}+\left(1-\limsup_{o\neq y\to o}\; \frac{V(y)}{\,-K_{d-2}(o,y)}\right)\cdot \delta_o.
\end{equation}

Let $o\in \Int Q=Q\Subset O$. The restriction of $\mathcal P_o$ to the class 
\begin{equation}\label{interAi0}
\begin{split}
\bigl\{\mu \in AS_o(O)\colon \supp \mu \cap Q=\varnothing \bigr\}
\\
\Bigl(\text{resp. $\bigl\{\mu \in J_o(O)\colon \supp \mu \cap Q=\varnothing \bigr\}$}\Bigr)
\end{split}
\end{equation}
define an affine bijection from class \eqref{interAi0}   onto class\/
{\rm (see   \eqref{ASP1}, Examples \ref{expASP}, \ref{expJP})}
\begin{equation}\label{PVl0}
\begin{split}
ASP^1(O\!\setminus\!o) \bigcap \har (Q\!\setminus\!o) 
\\
\Bigl(\text{resp. $JP_o^1(O) \bigcap \har \bigl(Q\!\setminus\!o\bigr)$}\Bigr).
\end{split}
\end{equation}
The restriction of $\mathcal P_o$ to the class 
\begin{equation}\label{interAi}
\begin{split}
\bigl\{\mu \in AS_o(O)\colon \supp \mu \cap Q=\varnothing \bigr\}
\bigcap  \bigl(C^{\infty}(O) \dd \lambda_d \bigr)\\
\Bigl(\text{resp. $\bigl\{\mu \in J_o(O)\colon \supp \mu \cap Q=\varnothing \bigr\}\bigcap  \bigl(C^{\infty}(O) \dd \lambda_d \bigr)$}\Bigr)
\end{split}
\end{equation}
define an affine bijection from class \eqref{interAi}  onto class
\begin{equation}\label{PVl}
\begin{split}
ASP^1(O\!\setminus\!o) \bigcap \har (Q\!\setminus\!o) \bigcap C^{\infty} (O\!\setminus\!o)\\  
\Bigl(\text{resp. $JP^1(O\!\setminus\!o)\bigcap \har (Q\!\setminus\!o) \bigcap C^{\infty} (O\!\setminus\!o) $}\Bigr).
\end{split}
\end{equation}
\end{dualtheo}
 This transition from the main bijection $\mathcal P_o$ to the bijection
from \eqref{interAi0} onto \eqref{PVl0} or 
 from \eqref{interAi} onto  \eqref{PVl} by restriction of $\mathcal P_o$  to \eqref{interAi0} or \eqref{interAi}  is quite obvious.

\begin{PJf}[{\rm \cite[Proposition 1.2]{Kha03}}] If $u\in \sbh(D)$, $u(o)\neq -\infty$, then
\begin{equation*}
u(o)=\int_Du\dd \mu-\int_{D\!\setminus\!o}\pt_{\mu-\delta_o} \dd \varDelta_u
\quad\text{for each  $\mu\in AS_o(D)$.}
\end{equation*}
\end{PJf}

\section{Embeddings of Arens\,--\,Singer\,/Jensen potentials into classes of test functions}\label{ASJem}
\setcounter{equation}{0}

\underline{Throughout this  Sec.~\ref{ASJem},} the boundary $\partial D$ of $D\ni o$ is {\it non-polar,\/} i.e., $\text{Cap}^*(\partial D)>0$. 

\begin{proposition}[{\rm a variant of  Phr\'agmen\,--\,Lindel\"of principle}]\label{FLp} If $v\in \sbh (D\!\setminus\!o)$
satisfies the conditions
\begin{equation}\label{limgv}
\limsup_{o\neq x\to o}\;\frac{v(x)}{\,-K_{d-2}(x,o)} \leqslant 0, 
\quad \limsup_{D\ni x\to \partial D} v(x) \leqslant 0,
\end{equation} 
then $v\leq 0$ on $D\!\setminus\!o$. In particular,
if a function $V\in \sbh (D\!\setminus\!o)$
satisfies the conditions
\begin{equation}\label{limgV}
\limsup_{o\neq x\to o}\frac{V(x)}{-K_{d-2}(x,o)} \leqslant c\in \RR^+, 
\quad \limsup_{D\ni x\to \partial D} V(x) \leqslant 0,
\end{equation}
then $V\leq cg_D(\cdot, o)$ on $D\!\setminus\!o$. 
\end{proposition}
\begin{proof} By conditions \eqref{limgv}, for any $a\in \RR^+\!\setminus\!0$, we have
\begin{equation*}
v(x)-ag_D(x,o)\leq O(1), \; o\neq x\to o;\quad
\limsup_{D\ni x\to \partial D} \bigl(v(x)-ag_D(x,o)\bigr)\leqslant 0. 
\end{equation*}
Hence the function $v-ag_D\in \sbh(D\!\setminus\!o)$ has the removable singularity at the point $o$ \cite[Theorem 5.16]{HK}, and the function 
\begin{equation}\label{vo}
\begin{cases}
v-ag_D &\quad\text{on $D\!\setminus\!o$},\\ 
\limsup\limits_{o\neq x\to o}\big(v(x)-ag_D(x,o)\big) 
&\quad\text{at $o$,}
\end{cases}
\end{equation}
is subharmonic on $D$ and
\begin{equation*}
 \limsup_{D\ni x\to \partial D}\big(v(x)-ag_D(x,o)\big)\leqslant 0. 
\end{equation*}
By the maximum principle, the function \eqref{vo} is negative on $D$, and $v\leq ag_D(\cdot, o)$
on $D\!\setminus\!o$ \textit{for an arbitrary\/} $a>0$. 
Thus, $v\leq 0$ on $D$. In particular, for $v:=V-cg_D(\cdot, o)$, under the conditions \eqref{limgV}, we have \eqref{limgv}
and obtain $V-cg_D(\cdot, o)\leq 0$ on $D$.
\end{proof}
\begin{theorem}[{\rm on embedding}]\label{l2} 
Let $S_o$ be a subset from\/ \eqref{S0seto}, and let $r, b_{\pm}$ are constants from\/  \eqref{bbpmr}. 
For any domain\/ $D_o$ satisfying 
\begin{equation}\label{o0SU}
o\in \Int S_o \Subset 
S_o^{\cup (4r)} \Subset D_o\Subset  D ,
\end{equation}
we can find a constant $B=\const^+_{o,S_o,r,D_o}\in \RR^+\setminus o$ 
such that 
\begin{subequations}\label{VvU0}
\begin{align}
ASP^1(D\!\setminus\!o)\bigcap \har (D_o\!\setminus\!o)
&\subset \sbh_{00}(D\!\setminus\!S_o ;r,-B< B),
\tag{\ref{VvU0}A}\label{{VvU0}AS}
\\
JP^1(D\!\setminus\!o)\bigcap \har (D_o\!\setminus\!o)
&\subset \sbh_{00}^+(D\!\setminus\!S_o;\leq B),
\tag{\ref{VvU0}J}\label{{VvU0}J}
\end{align}
\end{subequations}	
where, in the case\/ \eqref{{VvU0}AS}, 
we assume that the subset $S_o$ is connected.
\end{theorem}
\begin{proof} Let $V\in ASP^1(D\!\setminus\!o)\bigcap \har (D_o\!\setminus\!o)\supset JP^1(D\!\setminus\!o)\bigcap \har (D_o\!\setminus\!o)$.
\begin{lemma}\label{lemab}  If $V\in ASP(D\!\setminus\!o)$, then 
$V\leq g_D(\cdot, o)$ on $D$.
\end{lemma}
\begin{proof}[of Lemma\/ {\rm \ref{lemab}}]
The Arens\,--\,Singer potentials from Example \ref{expASP} satisfy   
conditions \eqref{limgV} of Proposition \ref{FLp}
with $c:=1$. Hence $V\leq g_D(\cdot, o)$ on $D$. 
\end{proof}
By Lemma \ref{lemab} we have
\begin{equation}\label{Babove}
\sup_{x\in S_o^{\cup(4r)}\!\setminus\!S_o}V(x)
\leqslant \sup_{x\in S_o^{\cup(4r)}\!\setminus\!S_o} g_D(x, o)=:B'=\const^+_{o,S_o,r}\in \RR^+\!\setminus\!o.
\end{equation} 
If $V\in  JP^1(D\!\setminus\!o)\bigcap \har (D_o\!\setminus\!o)$, then 
$V\geq 0$ on $\RR_{\infty}^d\!\setminus\!o$, and we obtain 
\eqref{{VvU0}J} with $B:=B'=\const^+_{o,S_o, r}$.
Otherwise, we use
\begin{lemma}\label{lembl} Under the conditions \eqref{o0SU},
 there is a constant $B''\in -\RR^+$ such that
\begin{equation}\label{pmulV}
\begin{split}
\inf_{x\in S_o^{\cup(4r)}\!\setminus\!o}V(x) &
\geqslant  B''=\const_{o,S_o,r,D_o}>-\infty
\\
\text{for every }V&
\in ASP^1(D\!\setminus\!o)
\bigcap \har (D_o\!\setminus\!o).
\end{split}
\end{equation} 
\end{lemma}
\begin{proof}[of Lemma\/ {\rm \ref{lembl}}] By Duality Theorem in version \eqref{interAi0}--\eqref{PVl0},  the Riesz measure  $\varDelta_V=c_d \bigtriangleup\! V$ is a Arens\,--\,Singer probability measure,    
$\varDelta_V\in \Meas_{\comp}^{1+}(D\!\setminus\!D_o)$
and $V=\pt_{\varDelta_V-\delta_o}$. 
By Proposition \ref{pt_below} with $\mu\overset{\eqref{muLo}}{:=}\varDelta_V$ and  $L\overset{\eqref{muLo}}{:=}S_o^{\cup(4r)}$, 
we have
\begin{multline*}
\inf_{x\in S_o^{\cup(4r)}\!\setminus\!o}V(x)= \inf_{x\in S_o^{\cup(4r)}\setminus o} {\pt}_{\varDelta_V-\delta_o}(x)
\overset{\eqref{{pmul}o}}{\geqslant} \varDelta_V(\RR^d)k_{d-2}\bigl(\dist(S_o^{\cup(4r)},\supp \varDelta_V)\bigr)-
k_{d-2}\left(\sup_{x\in S_o^{\cup(4r)}}|x|+|o|\right)
\\
\geqslant k_{d-2}\bigl(\dist(S_o^{\cup(4r)}\!\setminus\!o,\partial D_o)\bigr)+\const_{o,S_o,r}
\overset{\eqref{o0SU}}{=}
\const_{o,S_o,r,D_o}=:B''\overset{\eqref{o0SU}}{>}-\infty,
\end{multline*}
and we obtain  \eqref{pmulV}.
\end{proof}
If we set $B\overset{\eqref{Babove}}{:=} \max \{B', (-B'')^+\}$, then  \eqref{{VvU0}AS} follows from \eqref{Babove} and  \eqref{pmulV}.
\end{proof}

\begin{proof}[of implications\/ {\rm  [s\ref{{s}3}]$\Rightarrow$[s\ref{s9}]} and\/ {\rm [h\ref{{h}3}]$\Rightarrow$[h\ref{h8}]}] 
There is a domain $D_o$ satisfying \eqref{o0SU}. 
According to \eqref{S0seto}, we can choose a point $o\in \Int S_o$ such that  $u(o)\neq -\infty$. The latter means that 
\begin{equation*}
-\infty <\int_{S_o} k_{d-2}\bigl(|x-o|\bigr) \dd \varDelta_u(x)  
\overset{\text{i.e.}}{\Longrightarrow} \int_{S_o} g_D(x,o) \dd \varDelta_u(x)<+\infty.  
\end{equation*}
Thus, by Lemma \ref{lemab}, we obtain
\begin{equation}\label{Vu}
\int_{S_o} V \dd \varDelta_u\leqslant \int_{S_o} g_D(\cdot, o)
\dd \varDelta_u=:C_1=\const^+_{o,S_o,u}\in \RR^+    \quad
\text{for each $V\in ASP(D\!\setminus\!o)$}.
\end{equation}
Besides, by Lemma \ref{lembl}, we have 
\begin{equation}\label{VM}
-\infty <\const_{o,S_o,r, D_o,M}=B''\varDelta_M(S_o)\leqslant \int_{S_o} V\dd \varDelta_M
\quad
\text{for each $V\in ASP^1(D\!\setminus\!o)\bigcap \har (D_o)$}.  
\end{equation}
The conditions {[{s}\ref{{s}3}]} or  {[{h}\ref{{h}3}]} means that
there is a constant $C_2\in \RR$ such that  
\begin{subequations}\label{V2O}
\begin{align}
&\int_{D\!\setminus\!S_o} v \dd \varDelta_u
\leqslant \int_{D\!\setminus\!S_o}v\dd \varDelta_M +C_2\quad
\text{for each }v\overset{\eqref{VvU0}}{\in}
\left[\begin{array}{cc}
\mathcal V_b^+ &\text{in the case {[{s}\ref{{s}3}]}}, 
\tag{\ref{V2O}{}}\label{V2}
\\ 
 \mathcal V_b &\text{in the case {[{h}\ref{{h}3}]}},
\end{array} 
\right.\\
&\mathcal V_b^+:=\sbh_{00}^+(D\!\setminus\!S_o;\leq b)\bigcap  C^{\infty}(D\!\setminus\!S_o),\quad b:=b_+,
\notag
\\
&\mathcal V_b :=\sbh_{00}(D\!\setminus\!S_o; r,-b< b)\bigcap  C^{\infty}(D\!\setminus\!S_o), \quad
b:=\min \{b_+,-b_-\}.
\notag
\end{align}
\end{subequations}
Let $B\in \RR^+\!\setminus\!0$  be a constant from 
Theorem  \ref{l2}  on embedding with inclusions 
\eqref{VvU0}. We multiply both sides of inequality  \eqref{V2} by the number $B/b$ and obtain
\begin{equation*}
\int_{D\!\setminus\!S_o} v \dd \varDelta_u
\leqslant \int_{D\!\setminus\!S_o}v\dd \varDelta_M +BC_2
\text{ for each  }v\overset{\eqref{VvU0}}{\in}
\left[\begin{array}{cc}
\mathcal V_B^+ &\text{in the case {[{s}\ref{{s}3}]}},\\   
 \mathcal V_B &\text{in the case {[{h}\ref{{h}3}]}}.
\end{array} 
\right.
\end{equation*}
Hence, by inclusions  \eqref{VvU0} from Theorem \ref{l2}  on embedding, we have 
an inequality 
\begin{equation*}
\int_{D\!\setminus\!S_o} V \dd \varDelta_u
\leqslant \int_{D\!\setminus\!S_o}V\dd \varDelta_M +BC_2
\quad \text{for each  }V\overset{\eqref{VvU0}}{\in}
\left[\begin{array}{cc}
JP^1(D\!\setminus\!o)\bigcap \har (D_o\!\setminus\!o) \bigcap C^{\infty} (D\!\setminus\!o)&\text{in the case [s\ref{{s}3}]},\\   
ASP^1(D\!\setminus\!o)\bigcap \har (D_o\!\setminus\!o) \bigcap C^{\infty} (D\!\setminus\!o)&\text{in the case [h\ref{{h}3}]},
\end{array} 
\right.
\end{equation*}
which together with \eqref{Vu} (resp. \eqref{VM}) gives an  inequality
\begin{equation}\label{{V2Bin+}M}
%%\begin{align}
\int_{D\!\setminus\!o} V \dd \varDelta_u
\leq \int_{D\!\setminus\!\{o\}}V\dd \varDelta_M +C, 
\text{ for each  }V\overset{\eqref{VvU0}}{\in}
\left[\begin{array}{cc}
JP^1(D\!\setminus\!o)\bigcap \har (D_o\!\setminus\!o) \bigcap C^{\infty} (D\!\setminus\!o)&\text{in the case [s\ref{{s}3}]},\\   
ASP^1(D\!\setminus\!o)\bigcap \har (D_o\!\setminus\!o) \bigcap C^{\infty} (D\!\setminus\!o)&\text{in the case [h\ref{{h}3}]},
\end{array} 
\right.
%%\end{align}
\end{equation}
where  $C:=C_1+BC_2-B''\varDelta_M(S_o)=\const_{o,S_o,r,D_o,u,M}\in \RR$ is a constant independent of $V$. 
By Definition \ref{bal} in the form \eqref{SQ}, \eqref{{V2Bin+}M}
means that the measure $\varDelta_M$ is an affine $\bigl(JP^1(D\!\setminus\!o)\bigcap \har (D_o\!\setminus\!o) \bigcap C^{\infty} (D\!\setminus\!o)\bigr)$-balayage of the measure $\varDelta_u$ in the case {[\rm {s}\ref{{s}3}]}
or an affine  $\bigl(ASP^1(D\!\setminus\!o)\bigcap \har (D_o\!\setminus\!o) \bigcap C^{\infty} (D\!\setminus\!o)\bigr)$-balayage of the measure $\varDelta_u$  in the case {[\rm {h}\ref{{h}3}]}, respectively. 
The implications 
[s\ref{{s}3}]$\Rightarrow$[s\ref{s9}] and  [h\ref{{h}3}]$\Rightarrow$[h\ref{h8}] are proved.
\end{proof}

\begin{proof}[of implications\/ {\rm  [s\ref{s9}]$\Rightarrow$[s\ref{{s}1Jinfty}]} and\/ {\rm [h\ref{h8}]$\Rightarrow$[h\ref{{h}1Jinfty}]}]
Under the statement [s\ref{s9}] (resp. [h\ref{h8}]), there is a constant $C\in \RR$ such that inequality \eqref{{V2Bin+}M} is fulfilled for each 
potential  
\begin{equation}\label{VAJ}
V\overset{\text{or}}{\in} \left[\begin{array}{c}
ASP^1(D\!\setminus\!o)\bigcap \har (D_o\!\setminus\!o) \bigcap C^{\infty} (D\!\setminus\!o),\\ 
JP^1(D\!\setminus\!o) \bigcap \har (D_o\!\setminus\!o) \bigcap C^{\infty} (D\!\setminus\!o),
\end{array}\right. \quad \text{respectively.}
\end{equation} 
We choose $o\in S_o:=D_o$ so that $u(o)\neq -\infty$. 
By the Duality Theorem and the generalized Poisson\,--\,Jensen formula from  Sec.~\ref{ASP}, we have two equalities
\begin{equation}\label{uMV}
\begin{split}
\int_{D\!\setminus\!o} V \dd \varDelta_u&=u(o)+\int_D u \dd \varDelta_V , 
\quad
\int_{D\!\setminus\!o} V \dd \varDelta_M=M(o)+
\int_D M  \dd  \varDelta_V
\\
\text{for each  }V&\in ASP^1(D\!\setminus\!o)\bigcap \har (D_o\!\setminus\!o) \bigcap C^{\infty} (D\!\setminus\!o),
\end{split}
\end{equation}
where, in view of affine bijection $\mathcal P_o$ from \eqref{interAi} onto \eqref{PVl}, the Riesz measures $\quad \varDelta_V\overset{\eqref{df:cm}}{:=}c_d \bigtriangleup\!V$
of potentials $V$ from \eqref{VAJ} run through all Arens\,--\,Singer (resp. Jensen) measures in 
\begin{equation}\label{interAio}
\left[\begin{array}{c}
\mathcal M_{AS}^{\infty}  (D\!\setminus\!D_o):=AS_o(D)\bigcap  \Meas^{+1}_{\comp}(D\!\setminus\!D_o)\bigcap \bigl(C^{\infty}(D)\dd \lambda_d\bigr),\\ 
\mathcal M_{J}^{\infty}  (D\!\setminus\!D_o):=J_o(D)\bigcap  \Meas^{+1}_{\comp}(D\!\setminus\!D_o)\bigcap \bigl(C^{\infty}(D)\dd \lambda_d\bigr).
\end{array}\right. \;\text{resp.}
\end{equation}
Using   \eqref{{V2Bin+}M} with \eqref{VAJ} and \eqref{uMV} with \eqref{interAio}, we obtain 
\begin{equation}\label{affbalo}
\begin{split}
\int_D u\dd \mu &\leq \int_D M\dd \mu+(C+M(o)-u(0)) \\
\quad\text{for each }\mu&\overset{\eqref{interAio}}{\in} \left[\begin{array}{cc}
 \mathcal M_{AS}^{\infty}  (D\!\setminus\!D_o)&
\text{for the case [h\ref{h8}]},\\ 
\mathcal M_{J}^{\infty}  (D\!\setminus\!D_o)&
\text{for the case  [s\ref{s9}]},
\end{array}\right., \text{respectively,}       
\end{split} 
\end{equation}
where the constant $(C+M(o)-u(0))$ is independent of $\mu$. Thus, by Definition \ref{bal} in the form \eqref{MQ}, we get $u\curlyeqprec_{\mathcal M_{AS}^{\infty}  (D\!\setminus\!D_o)}M$ (resp. $u\curlyeqprec_{\mathcal M_{J}^{\infty}  (D\!\setminus\!D_o)}M$). This is exactly [h\ref{{h}1Jinfty}] (resp. [s\ref{{s}1Jinfty}]) for $S_o:=D_o$. 
\end{proof}

\section{ From Arens\,--\,Singer and Jensen  measures to subharmonic functions}\label{proofs31}
\setcounter{equation}{0}

Our final part of the proof of Theorems \ref{crit1} and \ref{crit2} uses the following corollary from the more general results of \cite{KhaRozKha19}.

\begin{theorem}[{\rm \cite[Ch.~2, {\bf 8.2}, Corollary 8.1, II, 1, (i)-(ii)]{KhaRozKha19}}]\label{Th5} Let $D\subset \RR^d$ be a domain, $H$ be a convex cone in $\sbh_*(D)$  containing constants, $D_o\Subset D$ be a subdomain,  
$o\in D_o$. 
Suppose that one of the following two conditions is fulfilled:
\begin{enumerate}[{\rm (a)}]
\item\label{sla} 
for any locally bounded from above sequence of functions  $( h_k)_{k\in \NN} \subset H$, the upper semicontinuous regularization of the upper limit 
$\limsup_{k\to\infty} h_k$ belongs to $H$ provided that $\limsup_{k\to\infty} h_k(x)\not\equiv -\infty$ on  $D$;
\item\label{slb} 
 $H$ is sequentially closed in $L^1_{\loc}(D)$.
\end{enumerate}
 Let $u\in \sbh_*(D)$, $M\in C(D)$, and 
$0\neq \vartheta \in \Meas^+_{\comp}(D)$, $\supp \vartheta \subset D_o\Subset D$.  If there is a constant $C\in \RR$ such that
\begin{equation}\label{affbal}
\begin{split}
\int_D u\dd \mu \leq \int_D M\dd \mu+C
\quad\text{for each $\mu\in \mathcal M^{\infty}(D\!\setminus\!D_o)$such that $\vartheta\preceq_H \mu$},  \\      
\text{where }\mathcal M^{\infty}(D\!\setminus\!D_o):=\Meas^+_{\comp}(D\!\setminus\!D_o)\bigcap \bigl(C^{\infty}(D)\dd \lambda_d\bigr),
\end{split}
\end{equation} 
then there is a function $h\in H$ such that $u+h\leq M$ on $D$.
\end{theorem}

\begin{proof}[of implications 
{\rm [h\ref{{h}1Jinfty}]$\Rightarrow$[h\ref{{h}1}]} and\/  {\rm [s\ref{{s}1Jinfty}]$\Rightarrow$[s\ref{{s}1}]}] We choose the convex cone 
\begin{equation*}
H:=\left[ \begin{array}{cc}
 \har(D)& \text{in the case {\rm  [h\ref{{h}1Jinfty}]} satisfying the condition \eqref{slb} of Theorem \ref{Th5},} \\ 
  \sbh_*(D)& \text{in the case {\rm [s\ref{{s}1Jinfty}]} satisfying the condition \eqref{sla} of Theorem \ref{Th5}},
\end{array} \right.
\end{equation*}
respectively. The statement {\rm  [h\ref{{h}1Jinfty}]} (resp. {\rm [s\ref{{s}1Jinfty}]}) means that we have \eqref{affbal} for $\vartheta:=\delta_o$ if we choose a domain $D_o$ so that $S_o\Subset D_o\Subset D$. 
By Theorem \ref{Th5}, there is a function $h\in \har(D)$ (resp. $h\in \sbh_*(D)$) such that $u+h\leq M$ on $D$. Thus, $u\prec_H M$, and 
Theorem \ref{crit2} (resp. Theorem \ref{crit1}) is proved.
\end{proof}

%\begin{acknowledgements}
%\end{acknowledgements}

\section*{Conflict of interest}
The authors declare that they have no conflict of interest.



\end{document}